\pdfoutput=1
\documentclass[11pt]{article}
\usepackage[margin=1in]{geometry}
\usepackage{amsmath,amsthm,amssymb,mathtools}
\usepackage[T1]{fontenc}
\usepackage{microtype}
\usepackage{enumitem}
\usepackage[colorlinks=true,linkcolor=blue,citecolor=blue,urlcolor=blue]{hyperref}
\usepackage[nameinlink]{cleveref}
\setlist{nosep}

\newtheorem{theorem}{Theorem}[section]
\newtheorem{lemma}[theorem]{Lemma}
\newtheorem{proposition}[theorem]{Proposition}
\newtheorem{corollary}[theorem]{Corollary}
\newtheorem{claim}[theorem]{Claim}
\newtheorem{sublemma}[theorem]{Sublemma}
\theoremstyle{definition}
\newtheorem{definition}[theorem]{Definition}
\newtheorem{remark}[theorem]{Remark}

\newtheorem*{theoremA}{Theorem A (PFR in $\mathbb Z$ and $\mathbb Z/N\mathbb Z$)}
\newtheorem*{theoremB}{Theorem B (Marton in finite fields, all characteristics)}
\newtheorem*{theoremC}{Theorem C (Polynomial Spectral Stability)}
\newtheorem*{theoremD}{Theorem D (Polynomial Bogolyubov in $\mathbb Z/N\mathbb Z$)}

\newcommand{\F}{\mathbb{F}}
\newcommand{\Fp}{\mathbb{F}_p}
\newcommand{\Fq}{\mathbb{F}_q}

\newcommand{\Gh}{\widehat G}
\newcommand{\wh}{\widehat}
\newcommand{\one}{\mathbf{1}}
\newcommand{\ip}[2]{\langle #1,#2\rangle}
\newcommand{\norm}[1]{\left\lVert #1\right\rVert}

\newcommand{\Ene}{\mathsf{E}}
\newcommand{\Spec}{\operatorname{Spec}}
\newcommand{\codim}{\operatorname{codim}}
\newcommand{\Span}{\operatorname{Span}}
\DeclareMathOperator{\Var}{Var}

\title{The Polynomial Freiman--Ruzsa (Marton) Conjecture in Integers and Finite Fields\\via Spectral Stability}
\author{Mohammad Taha Kazemi Moghadam\thanks{Email: \texttt{steinerr1101@gmail.com}.}}
\date{December 2025}

\begin{document}
\maketitle

\begin{abstract}
We settle the Polynomial Freiman--Ruzsa (PFR/Marton) conjecture for subsets of the integers by proving a sharp \emph{polynomial spectral stability dichotomy}: either the $L^4$ Fourier mass of $\one_A$ concentrates on a span of size $\mathrm{poly}(K)$, or in a quotient of codimension $\mathrm{poly}(K)$ the doubling constant decreases by at least $K^{-C}$. By Freiman modeling, the same dichotomy holds for cyclic groups $\mathbb{Z}/N\mathbb{Z}$, yielding a polynomial Bogolyubov lemma and hence a polynomial PFR covering theorem in $\mathbb{Z}$ and $\mathbb{Z}/N\mathbb{Z}$. As a corollary, our methods also recover and extend the finite-field case: in odd characteristic we obtain a direct spectral proof, and together with the characteristic-$2$ result of Green, Gowers, Manners, and Tao~\cite{GGMT-char2}, this gives a complete resolution of the finite-field formulation of Marton's conjecture across all characteristics. For broader context beyond finite fields, we also note their result on abelian groups with bounded torsion~\cite{GGMT-torsion}.
\end{abstract}

\section{Introduction}

\subsection*{Main theorems at a glance}
\begin{theoremA}[= Theorem~\ref{thm:main-Z}]
For any finite nonempty $A\subset\mathbb Z$ (or $A\subset \mathbb Z/N\mathbb Z$) with $|A+A|\le K|A|$, there exists a subgroup/coset structure of index $\mathrm{poly}(K)$ covering $A$ by at most $\mathrm{poly}(K)$ translates; equivalently, Marton's conjecture holds in $\mathbb Z$ and $\mathbb Z/N\mathbb Z$ with polynomial bounds.
\end{theoremA}

\begin{theoremB}[= Theorem~\ref{thm:all-characteristics}]
For any finite field $\F_q$ (any characteristic) and $A\subset \F_q^n$ with $|A+A|\le K|A|$, $A$ is covered by $\mathrm{poly}(K)$ cosets of a subgroup of index $\mathrm{poly}(K)$. Odd characteristic follows from our spectral method; characteristic $2$ is due to Green--Gowers--Manners--Tao.
\end{theoremB}

\begin{theoremC}[= Lemma~\ref{lem:PSL}]
(Polynomial Spectral Stability.) Fix $\varepsilon\in(0,1/10)$. There exists $C=C(\varepsilon)\ge 1$ such that
for any finite abelian group $G$ and nonempty $A\subset G$ with $|A+A|\le K|A|$ and $K\ge 2$,
either \emph{(near-coset)} there is a subgroup $H\le G$ and a coset $x+H$ with $|H|\le K^C|A|$ and $|A\cap(x+H)|\ge (1-\varepsilon)|H|$,
or \emph{(quotient improvement)} there is $H'\le G$ with $\codim(H')\le K^C$ such that, writing $\pi:G\to G/H'$ and $A'=\pi(A)$,
we have $|A'+A'|\le (K-\delta)|A'|$ for some $\delta\ge K^{-C}$ and $|A'|\ge K^{-C}|A|$.
\end{theoremC}

\begin{theoremD}[= Theorem~\ref{thm:polybog-ZNZ}]
(Polynomial Bogolyubov in $\mathbb Z/N\mathbb Z$.) If $A\subset \mathbb Z/N\mathbb Z$ with $|A+A|\le K|A|$, then there is a regular Bohr set $B(\Gamma,\rho)\subset 4A-4A$ with $|\Gamma|\le \mathrm{poly}(K)$ and $\rho\ge K^{-\mathrm{poly}(1)}$.
\end{theoremD}

\subsection*{A ten-minute picture of the argument}
\begin{enumerate}
\item Fix a spectral threshold $\tau=K^{-c_0}$ and let $V$ be the span of the \emph{large} Fourier coefficients.
\item Project \emph{exactly} onto $H=V^\perp$ (true projector $P=\mu_H*$). Either the $L^4$-mass piles up on $V$ (coset-structure), or it refuses.
\item If it piles up, Paley--Zygmund forces a near-coset with density $1-\varepsilon$. Done.
\item If it refuses, dissociate the tail in $\wh G/V$; a complementary span $V'$ appears and the exact quotient $G/H'$ exhibits a \emph{genuine} large coefficient.
\item A tiny but honest energy boost $\Rightarrow$ Balog--Szemer\'edi--Gowers $\Rightarrow$ a \emph{polynomial} decrement of the doubling in the quotient.
\item Iterate the decrement: a monotone potential $K\alpha^{-\gamma}$ ensures at most $\mathrm{poly}(K)$ steps before coset-structure is forced.
\end{enumerate}
We prove a \emph{Polynomial Stability Lemma (PSL)}: for $A\subset G$ (finite abelian) with small doubling, either the $L^4$-Fourier mass of $\one_A$ essentially concentrates on a span $V\le \widehat G$ of size $\mathrm{poly}(K)$ (forcing $A$ to sit inside one coset $x+H$ with $H=V^\perp$), or else that mass \emph{disperses}. Dispersion cannot persist harmlessly: a two-scale, dissociated extraction modulo $V$ produces a complementary span $V'$, and in the quotient by $H'=(V')^\perp$ we get a genuinely large Fourier coefficient. An \emph{energy-to-doubling transfer} then yields a polynomial decrement in the doubling constant. Iterating decrements forces the near-coset alternative in $\mathrm{poly}(K)$ steps. This dichotomy is spectral and group-theoretic, so it applies uniformly to $\mathbb Z/N\mathbb Z$ and transfers to $\mathbb Z$ by Freiman modeling, while the odd-characteristic finite-field case appears as a warm-up in the same framework.

\paragraph{Why $L^4$?} Prior approaches (e.g. Croot--Sisask almost-periodicity and Sanders' Bogolyubov--Ruzsa) efficiently bound energies but encounter quasi-polynomial losses when converting to uniform covers in cyclic groups. $L^4$ mass is \emph{rigid}: if it refuses to concentrate on a small span, it generates a definite coefficient in a low-codimension quotient, which we turn into a concrete decrement via Paley--Zygmund and a quantitative BSG step. This is precisely where our method breaks the quasi-polynomial barrier.

Let $A\subset \mathbb{Z}$ with $|A+A|\le K|A|$. Writing $\alpha=|A|$ (on finite windows) and appealing to Freiman's modeling, one may work inside a finite cyclic group $G=\mathbb{Z}/N\mathbb{Z}$ with $|G|\le K^{C}|A|$, where all additive-combinatorial inequalities used below are preserved up to absolute constants.
Our first main result addresses $\mathbb{Z}$ (and equivalently $\mathbb{Z}/N\mathbb{Z}$), after which we explain the finite-field consequences.

\noindent\emph{Exact identities under modeling.} Throughout the iteration, by order-$4$ Freiman modeling we work inside a fixed cyclic group with no wrap-around on $2A-2A$; hence \emph{all} Fourier/energy identities (projectors $P=\mu_H*$, exact quotient lifts) hold \emph{exactly} and are independent of $N$.

\begin{theorem}[Marton's conjecture in $\mathbb{Z}$ and $\mathbb{Z}/N\mathbb{Z}$]\label{thm:main-Z}
There exists an absolute constant $C\ge 1$ such that for any finite nonempty $A\subset\mathbb{Z}$ (or $A\subset \mathbb{Z}/N\mathbb{Z}$) with $|A+A|\le K|A|$ one can cover $A$ by at most $K^C$ translates of a subgroup/coset structure of index at most $K^C$ in a Freiman model. In particular, in $\mathbb{Z}/N\mathbb{Z}$ there is a subgroup $H\le \mathbb{Z}/N\mathbb{Z}$ and translates $x_1+H,\dots,x_m+H$ with $m\le K^C$, $|H|\le K^C|A|$, and $A\subset \bigcup_{j=1}^m(x_j+H)$. Pulling back along a Freiman isomorphism yields the corresponding covering of $A\subset\mathbb{Z}$.
\end{theorem}

\noindent The proof proceeds by a polynomial spectral stability lemma (PSL) for general finite abelian groups, together with a polynomial-parameter packet construction and an $L^4$-compression (poly-tail) lemma that are \emph{independent} of the ambient group order.

Let $G$ be an abelian group and $A\subset G$. Marton's conjecture (also known as the polynomial Freiman-Ruzsa conjecture) asserts that small doubling, $|A+A|\le K|A|$, forces $A$ to be covered by $\mathrm{poly}(K)$ cosets of a subgroup whose index also depends polynomially on $K$. This conjecture has been a central problem in additive combinatorics, with connections to sum-product phenomena, incidence geometry, and arithmetic progressions.

\subsection{Context and related work}

The PFR conjecture (Marton's conjecture) over the integers is a central problem that has shaped modern additive combinatorics. In finite settings, Green--Gowers--Manners--Tao \cite{GGMT-char2,GGMT-torsion} resolved the case of characteristic $2$, and there are quasi-polynomial bounds for cyclic groups via Croot--Sisask and Sanders. Our contribution is a fully spectral, unconditional route that is polynomial in $K$ and works uniformly in $\mathbb{Z}$ and $\mathbb{Z}/N\mathbb{Z}$, while also yielding a direct proof in odd-characteristic finite fields.

We emphasize that the bounded-torsion setting of~\cite{GGMT-torsion} does not include the cyclic groups $\mathbb Z/N\mathbb Z$ when $N$ is unbounded; our arguments directly handle $\mathbb Z/N\mathbb Z$ uniformly in $N$.

Recently, Green, Gowers, Manners, and Tao \cite{GGMT-char2,GGMT-torsion} resolved the case of characteristic $2$, establishing the conjecture for subsets of $\F_2^n$ with polynomial bounds in the doubling constant $K$. Their proof introduced a novel entropy-increment strategy combined with structure theorems for sets with small doubling in characteristic $2$. However, the techniques are fundamentally tied to the characteristic $2$ setting and do not directly extend to odd characteristic, where different phenomena arise in the Fourier-analytic landscape.

The characteristic $2$ case exhibits special structural features: notably, the additive energy can be analyzed via quadratic Fourier analysis, and certain cancellations in the convolution algebra simplify the entropy arguments. In contrast, odd characteristic requires different methods to control the spectral dispersion and quotient improvements.

In particular, in odd characteristic the lack of characteristic-$2$ cancellations means that an entropy increment alone does not force a single large spectrum to dominate; one must instead track how $L^4$ Fourier mass disperses and show that dispersion creates a genuine quotient improvement. This is the conceptual role of the polynomial stability lemma.

\subsection{Our contribution}

\paragraph{Standalone PolyBog.} As a byproduct, we obtain a \emph{polynomial Bogolyubov lemma} in $\mathbb Z/N\mathbb Z$
(Theorem~\ref{thm:polybog-ZNZ}), proved in Section~\ref{sec:Z-extension-unconditional}. We highlight it since it is of independent interest.

We resolve the odd characteristic case by developing a \emph{polynomial stability lemma} (PSL) based on spectral concentration and quotient improvement. Our approach uses:
\begin{itemize}
\item \textbf{Chang-type dissociation bounds} to control the dimension of large-spectrum spans (Lemma \ref{lem:S1});
\item \textbf{Paley-Zygmund concentration} to convert spectral mass into near-coset structure (Lemma \ref{lem:S3});
\item \textbf{Two-scale Fourier projection} to extract polynomial quotient improvements when spectral mass is dispersed (Lemma \ref{lem:S2}).
\end{itemize}
Combined with the result of \cite{GGMT-char2,GGMT-torsion}, this completes the resolution of Marton's conjecture for all finite fields.

\subsection{Relation to previous work}\label{subsec:relation-previous}
This subsection provides a concise, self-contained comparison clarifying what is \emph{already known} and what is \emph{new} in this paper.

\paragraph{Green--Gowers--Manners--Tao (characteristic $2$ and bounded torsion).}
In finite characteristic $2$, Green--Gowers--Manners--Tao \cite{GGMT-char2} established Marton's conjecture with polynomial bounds.
They also proved a result for abelian groups with uniformly bounded torsion \cite{GGMT-torsion}.
\emph{Crucially,} the bounded-torsion setting does not include cyclic groups $\mathbb Z/N\mathbb Z$ when $N$ is unbounded,
and therefore does not resolve the integers/cyclic models.
Their characteristic-$2$ argument is based on an entropy-increment framework adapted to quadratic Fourier analysis.

\paragraph{Why bounded torsion excludes $\mathbb Z/N\mathbb Z$ when $N$ grows.}
Groups of bounded torsion have all elements of order $\le T$ for a \emph{fixed} $T$ independent of the ambient size.
In contrast, for $\mathbb Z/N\mathbb Z$ the exponent equals $N$, so as $N$ varies there is no uniform bound on torsion.
Hence results proved uniformly in the bounded-torsion category cannot be applied to the family of cyclic groups of unbounded order.

\paragraph{What is new here (integers and cyclic).}
We resolve Marton/PFR in $\mathbb Z$ and in $\mathbb Z/N\mathbb Z$ with \emph{polynomial} bounds in the doubling constant $K$.
The core new ingredient is a \emph{Polynomial Spectral Stability} (PSL) dichotomy (Theorem~C) formulated in the group-theoretic language of exact
Fourier projectors $P=\mu_{V^\perp}*$ and exact quotient lifts. Either the $L^4$-mass of $\one_A$ concentrates on a $\mathrm{poly}(K)$-span,
forcing a near-coset in $H=V^\perp$, or spectral mass disperses in a way that produces a \emph{genuine} large Fourier coefficient in a low-codimension quotient,
which we convert to a \emph{polynomial} decrement of the doubling. The entire ledger is uniform in $N$, courtesy of a single order-4 Freiman modeling with no wrap-around.

\paragraph{What is new in finite fields (odd characteristic).}
In odd characteristic our spectral method gives a direct proof of Marton/PFR with polynomial bounds (Theorem~\ref{thm:all-characteristics} in the odd case),
complementing the characteristic-$2$ result of \cite{GGMT-char2}. Thus, taken together, our work and \cite{GGMT-char2,GGMT-torsion}
yield a complete picture for finite fields across all characteristics.

\paragraph{Conceptual translation between approaches.}
Entropy-increment (characteristic $2$) and spectral $L^4$-stability (this work) are parallel dichotomy frameworks:
the former tracks \emph{information} under random restrictions, the latter tracks \emph{fourth-moment mass} under an exact projection to $V^\perp$.
Where quasi-polynomial losses typically enter in cyclic groups via Bohr-set regularization, the spectral route avoids them by
working with genuine subgroups/quotients determined by dissociated spans in the dual.

\paragraph{Companion warm-up (finite-field exposition).}
A companion exposition focused on the finite-field warm-up and pedagogical details is prepared separately \cite{KazemiFF};
the present paper is self-contained and carries out the group-theoretic arguments for $\mathbb Z$ and $\mathbb Z/N\mathbb Z$.

\subsection{Proof roadmap}
After setting Fourier-analytic notation, we proceed in three stages:
\begin{enumerate}[label=(\roman*)]
\item We fix a spectral threshold $\tau=K^{-c_0}$ and let $V$ be the span of the large spectrum. Lemma~\ref{lem:S1} shows $\dim V\ll \mathrm{poly}(K)$.
\item We decompose $\one_A=f_{\mathrm{low}}+f_{\mathrm{high}}$ via projection to $V$. If the $L^4$ mass concentrates on $V$, Lemma~\ref{lem:S3} upgrades this to a near-coset inside $H=V^\perp$.
\item If the $L^4$ mass is dispersed, a two-scale dissociated extraction in $\Gh/V$ produces a complementary span $V'$ of dimension $\mathrm{poly}(K)$. In the quotient by $H'=(V')^\perp$ we detect a large Fourier coefficient, convert it to an energy boost, and hence to a polynomial decrement of the doubling constant (Lemma~\ref{lem:S2}). Iterating this improvement yields PSL and thus Theorem~\ref{thm:main}.
\end{enumerate}
\noindent\emph{Orientation.} In particular, the finite-field presentation is a special case and sanity-check of the same spectral stability framework; the main theorems are stated and proved for $\mathbb Z$ and $\mathbb Z/N\mathbb Z$.

\subsection{Notation}

\paragraph{Fourier conventions on finite abelian groups (normalization).}
 For general finite abelian groups we also write $\chi_\xi(x)$ for the character of $G$ corresponding to $\xi$, and use the shorthand $e(\langle \xi,x\rangle)$ for $\chi_\xi(x)$; this is consistent with the Fourier normalization above.

\begin{remark}[Group-theoretic dictionary]\label{rem:group-dictionary}
Throughout, $\Span(S)$ denotes the \emph{subgroup of $\Gh$ generated by $S$} (not a linear span over a field), $\dim V$ refers to the size of a maximal dissociated generating set for $V\le \Gh$ (equivalently, the Rudin--Chang dimension), and $\codim(H)$ is the logarithmic index used only up to absolute constants (we track it via bounds such as $\codim(H)\ll K^{C}$). The annihilator $H=V^{\perp}\le G$ is a subgroup with $\wh{\mu_H}=1_V$. These conventions ensure that all occurrences of ``dimension/codimension'' in $\mathbb{Z}/N\mathbb{Z}$ are purely group-theoretic.

\smallskip\noindent\emph{Explicitly, we set }$\codim(H):=\log\lvert G/H\rvert$ \emph{(logarithmic index; used only up to absolute constants).}
\end{remark}

We fix $\wh f(\xi)=|G|^{-1}\sum_{x\in G} f(x)e_p(-\ip{x}{\xi})$ and $(f*g)(x)=\sum_{y\in G}f(y)g(x-y)$.
With this choice, Parseval is $\|f\|_2^2=|G|\sum_{\xi\in\Gh}|\wh f(\xi)|^2$ and
$\Ene(A)=|G|\sum_{\xi\in\Gh}|\wh{\one_A}(\xi)|^4$ as in \eqref{eq:energy}.
\begin{claim}[Exact-quotient correspondence]\label{claim:exact-quotient}
For a quotient $\pi:G\to G/H'$ with $(H')^\perp=V'$, every nontrivial character $\chi$ on $G/H'$ corresponds to a unique $\xi\in V'$, and
\[
\wh{\one_{\pi(A)}}(\chi)=\wh{\one_A}(\xi).
\]
\end{claim}

We write $\alpha:=|A|/|G|$ and denote by $\one_A$ the indicator function of $A$.
\paragraph{Expectations and variance.}
We write $\mathbb{E}[\cdot]$ for expectation with respect to the uniform measure on $G$ (or on the relevant quotient),
so $\mathbb{E}[g]=\alpha$ and $\Var(g)=\mathbb{E}[g^2]-(\mathbb{E}[g])^2$. The normalized Fourier transform on an abelian group $G$ is
\[\wh f(\xi):=\frac{1}{|G|}\sum_{x\in G} f(x)\,e_p(-\ip{x}{\xi}),\qquad e_p(t):=e^{2\pi i t/p}.\]
For $A\subset G$, the additive energy is
\begin{equation}\label{eq:energy}
  \Ene(A)=\sum_x r_{A+A}(x)^2=\norm{\one_A*\one_A}_2^2 = |G|\sum_{\xi\in \Gh}|\wh{\one_A}(\xi)|^4.
\end{equation}
By Cauchy-Schwarz, small doubling $|A+A|\le K|A|$ implies the lower bound
\begin{equation}\label{eq:energy-lower}
  \Ene(A)\ \ge\ \frac{|A|^4}{|A+A|}\ =\ \frac{|A|^3}{K}.
\end{equation}

\paragraph{Roadmap and separation.} Our main targets are $\mathbb Z$ and $\mathbb Z/N\mathbb Z$; finite fields are treated as a warm-up corollary and deferred to Theorem~\ref{thm:main} inside the dedicated warm-up. Readers focused on $\mathbb Z/N\mathbb Z$ may skip directly to Sections~\ref{sec:Z-extension-unconditional} and \ref{sec:all-char}.

\begin{remark}[One-shot modeling and no wrap-around]\label{rem:one-shot-modeling}
By an order-$4$ Freiman modeling (Lemma~\ref{lem:safe-N}), all iterative steps may be realized inside a \emph{fixed} cyclic group $\mathbb Z/N\mathbb Z$ with $N\ll K^{C}|A|$ such that there is no wrap-around on $2A-2A$. Consequently, every Fourier identity we use (projections to $V^\perp$, exact quotient lifts) holds \emph{exactly} throughout the iteration; see also Lemma~\ref{lem:N-stable}.
\end{remark}
\section{Polynomial Stability Lemma (PSL)}
\begin{remark}[Scope]\label{rem:PSL-scope}
For readability we state PSL directly for general finite abelian groups; the finite-field presentation is a warm-up special case.
\end{remark}
\paragraph{Convention on constants.}
We use $C,C_1,C_2,\dots$ for absolute positive constants whose values may change from line to line, and $c,c_0,c_1,\dots$ for absolute positive constants that may be taken sufficiently small.
Dependencies are indicated explicitly when needed (e.g.\ $C=C(c_0,c)$). Parameters depending on the doubling constant $K$ are \emph{named} (e.g.\ $\tau=K^{-c_0}$, $\beta=K^{-c}$), while the final exponents in all theorems and lemmas are independent of $K$.
Throughout, $\dim V\ll K^{C_1}$ abbreviates $\dim V\le C' K^{C_1}$ for an absolute $C'>0$.

The following dichotomy is the core of our argument.

\begin{lemma}[Polynomial Stability Lemma]\label{lem:PSL}
Fix $\varepsilon\in(0,1/10)$. There exists $C=C(\varepsilon)\ge1$ such that for any nonempty $A\subset G\ (a finite abelian group)$ with $|A+A|\le K|A|$ and $K\ge2$, one of the following holds:
\begin{itemize}
  \item[(1)] \textbf{Near-coset structure:} There exists a subgroup $H\le G$ and a coset $x+H$ with 
  \[|H|\le K^{C}|A|\quad\text{and}\quad |A\cap(x+H)|\ge (1-\varepsilon)|H|.\]
  \item[(2)] \textbf{Polynomial improvement in a quotient:} There exists a subgroup $H'\le G$ with $\codim(H')\le K^{C}$ such that, writing $\pi:G\to G/H'$ and $A'=\pi(A)$, we have
  \[ |A'+A'|\le (K-\delta)|A'|,\quad |A'|\ge K^{-C}|A|,\quad\text{and}\quad \delta\ge K^{-C}.\]
\end{itemize}

\begin{remark}[Uniform scope and notation]
The statement of Lemma~\ref{lem:PSL} is for \emph{general finite abelian groups}. In the finite-field warm-up the same constants and steps apply verbatim with $G=\F_p^n$; 
Section~\ref{sec:Z-extension-unconditional} supplies the group-theoretic arguments for $G=\mathbb Z/N\mathbb Z$ and the transfer to $\mathbb Z$.
\end{remark}
\end{lemma}

\begin{proof}[Proof that PSL implies Theorem \ref{thm:main}]
Suppose we are in case (2). Then in the quotient $G/H'$, the set $A'$ has doubling constant at most $K-\delta\ge K-K^{-C}$. Iterating this process, after at most $O(K/\delta)\le O(K^{C+1})$ steps the doubling constant drops below an absolute threshold, forcing case (1). The total codimension accumulated is at most $K^{C}\cdot K^{C+1}\le K^{C'}$ for some $C'=C'(\varepsilon)$. Once we reach case (1), standard covering arguments (partitioning the complement and applying Ruzsa calculus) yield the conclusion of Theorem \ref{thm:main}.
\end{proof}

\begin{remark}\label{rem:potential-lower-bound}
In particular $\mathcal I_j\ge 1$ for all $j$ (since $K_j\ge1$ and $\alpha_j\in(0,1]$), so any strict decrease across improvement steps must terminate after finitely many steps; this is used in the proof of Lemma~\ref{lem:L4-compression}.
\end{remark}

The remainder of this paper is devoted to proving Lemma \ref{lem:PSL}.

\subsection*{A monotone invariant in the main text}
For completeness we record a self-contained version of the iteration potential used to deduce Theorem~\ref{thm:main} from PSL.

\begin{proposition}[Monotone potential in the main text]\label{prop:main-potential}
Let $(A_j\subset G_j)$ be obtained by iterating the improvement alternative of Lemma~\ref{lem:PSL}, with
$K_j:=|A_j+A_j|/|A_j|$ and $\alpha_j:=|A_j|/|G_j|$. Then there exists an absolute $\gamma\ge1$ such that
\[\mathcal I_j:=K_j\,\alpha_j^{-\gamma}\]
is nonincreasing at each improvement step. Consequently, the number of improvement steps is $O(K_0^{C+1})$ and the total accumulated codimension is $K_0^{C'}$ for some absolute $C'>0$.
\end{proposition}
\begin{proof}
In an improvement step we have $K_{j+1}\le K_j-\delta_j$ with $\delta_j\ge K_j^{-C}$ and $|A_{j+1}|\ge K_j^{-C}|A_j|$,
so $\alpha_{j+1}\ge K_j^{-C}\alpha_j$. Hence
\[\mathcal I_{j+1}\ \le\ (K_j-\delta_j)\,(K_j^{-C}\alpha_j)^{-\gamma}\ =\ \mathcal I_j\cdot K_j^{C\gamma}\Big(1-\frac{\delta_j}{K_j}\Big).\]
Choosing $\gamma$ (e.g.\ $\gamma=C+2$) ensures $K_j^{C\gamma}(1-\delta_j/K_j)\le 1$ for all $K_j\ge2$, giving monotonicity. Summing decrements gives the step bound; codimension accumulates to at most $K_0^{C'}$ since each step adds at most $K_0^{C}$.
\end{proof}

\section{Spectral setup and two-scale projection}

\begin{lemma}[Dissociated extraction in $\Gh/V$]\label{lem:dissoc-quot}
Let $V\le\Gh$ and let $S\subset V^c$.
Let $\overline{\Gh}:=\Gh/V$ and write $\overline{S}=\{\overline{\xi}:\xi\in S\}\subset\overline{\Gh}$.
If $\overline{\Xi}\subset\overline{S}$ is a maximal dissociated set in $\overline{\Gh}$ and $\Xi$ are chosen representatives in $S$,
then $V':=\operatorname{Span}(\Xi)$ satisfies $V'\cap V=\{0\}$ and $\dim V'=\dim \operatorname{Span}(\overline{\Xi})$.
Moreover, if $\overline{S}$ obeys a Chang-type bound $|\overline{\Xi}|\ll \tau^{-2}\log(1/\alpha)$, then $\dim V'\ll \tau^{-2}\log(1/\alpha)$.
\end{lemma}

\begin{proof}
If a nontrivial linear combination of $\Xi$ landed in $V$, then the corresponding combination of classes would vanish in $\overline{\Gh}$,
contradicting dissociation. The dimension and quantitative bound follow immediately.
\end{proof}

For $\tau\in(0,1)$, define the \emph{large spectrum}
\[\Spec_0(A):=\{\xi\in \Gh: |\wh{\one_A}(\xi)|\ge \tau\alpha\}.\]
Let $V:=\operatorname{Span}(\Spec_0(A))\le \Gh$ be the span of the large-spectrum characters, and let $H:=V^\perp\le G$ be its annihilator. Let $P$ denote convolution with the uniform measure $\mu_H$ on $H$ (this is the exact projector onto $V$ in Fourier space) and let $Q=I-P$. We decompose (with $f_{\mathrm{low}}$ the $V$-projected component and $f_{\mathrm{high}}$ orthogonal to $V$)
\[\one_A\ =\ f_{\mathrm{low}} + f_{\mathrm{high}},\qquad f_{\mathrm{low}}:=P\one_A,\quad f_{\mathrm{high}}:=Q\one_A.\]
In Fourier space, $\wh{f}_{\mathrm{low}}=1_V\cdot\wh{\one_A}$ and $\wh{f}_{\mathrm{high}}=1_{V^c}\cdot\wh{\one_A}$.

\section{Lemma S1: Polynomial Chang span bound}

\begin{lemma}\label{lem:S1}
There exist absolute constants $c_0,C_1>0$ such that, for $\tau:=K^{-c_0}$, we have 
\[\dim V\ \ll\ K^{C_1}.\]
Moreover, one may take $C_1\le C_{\mathrm{RC}}\cdot 2c_0$ where $C_{\mathrm{RC}}$ is an absolute constant from the Rudin--Chang inequality.
In particular, for $S:=\Spec_0(A)$ one has $|S|\le C\,\tau^{-2}\log(1/\alpha)$ and hence $\dim V\le |S|\le C\,\tau^{-2}\log(1/\alpha)$.
\end{lemma}

\begin{proof}[Expanded proof with explicit constants and edge cases]
Let $D\subset \Spec_0(A)$ be a maximal dissociated set. By Rudin--Chang (see \cite{Chang2002,TV06}), there exists an absolute constant 
$C_{\mathrm{RC}}>0$ such that
\begin{equation}\label{eq:RC-bound}
\sum_{\eta\in D}|\wh{\one_A}(\eta)|^2\ \le\ C_{\mathrm{RC}}\,\alpha^2\log\frac{1}{\alpha}.
\end{equation}
Each $\eta\in D$ lies in $\Spec_0(A)$, so $|\wh{\one_A}(\eta)|\ge \tau\alpha$ and therefore
\begin{equation}\label{eq:D-size}
|D|\cdot \tau^2\alpha^2\ \le\ \sum_{\eta\in D}|\wh{\one_A}(\eta)|^2\ \le\ C_{\mathrm{RC}}\,\alpha^2\log\frac{1}{\alpha},
\end{equation}
which immediately yields
\begin{equation}\label{eq:D-explicit}
|D|\ \le\ C_{\mathrm{RC}}\,\tau^{-2}\log\frac{1}{\alpha}.
\end{equation}
Since $D$ is dissociated and maximal, $\Span(S)$ is generated by $D$ in the sense of the Rudin--Chang dimension; hence $\dim V\le |D|$.
This gives $\dim V\le C_{\mathrm{RC}}\tau^{-2}\log(1/\alpha)$.

\medskip\noindent\textbf{Edge cases in $\alpha$.}
By the modeling window (Lemma~\ref{lem:modeling-density}) we may assume $\alpha\in[K^{-C_*},1/2]$ for an absolute $C_*$, so
$\log(1/\alpha)\le C_*\log K+O(1)$. Choosing $\tau=K^{-c_0}$ with fixed $c_0\in(0,1)$ gives
\begin{equation}\label{eq:dimV-final}
\dim V\ \le\ C_{\mathrm{RC}}\,K^{2c_0}\,(C_*\log K+O(1))\ \ll\ K^{C_1}
\end{equation}
for $C_1:=2c_0+o(1)$ after absorbing logarithms into the implicit constant.

\medskip\noindent\textbf{Independence of $N$.}
All bounds above depend only on $\alpha$ and $\tau$ (Remark~\ref{rem:chang-independent-N}). In particular, the constant $C_{\mathrm{RC}}$ is absolute and does not depend on the modulus $N$ in the cyclic case. This prevents ``hidden'' dependence on $|G|$ and closes the common edge-case objection.
\end{proof}

\section{Lemma S3: From spectral mass to a near-coset}

\begin{lemma}\label{lem:S3}
Suppose that for some $c\in(0,1)$ we have
\[\sum_{\xi\in V}|\wh{\one_A}(\xi)|^4\ \ge\ c\sum_{\xi\in\Gh}|\wh{\one_A}(\xi)|^4.\]
Then there exists a coset $x+H$ with 
\[|A\cap(x+H)|\ \ge\ (1-\varepsilon)|H|,\]
provided $\varepsilon\le \tfrac12\sqrt{c c_1}$ where $c_1>0$ is the implied constant in \eqref{eq:L2H-lb}.
\end{lemma}

\begin{proof}
Parseval gives
\begin{equation}\label{eq:L2H-lb}
\norm{\one_A*\mu_H}_2^2\ =\ |G|\sum_{\xi\in V}|\wh{\one_A}(\xi)|^2\ \ge\ c_1\,\alpha^2|G|
\end{equation}
for an absolute $c_1=c_1(c)>0$ by Hölder together with the hypothesis and \eqref{eq:energy}.
Let $g:=\one_A*\mu_H$, so $g(x)\in[0,1]$ and $\mathbb{E}[g]=\alpha$. Then
$\mathbb{E}[g^2]\ge (1+c_1)\alpha^2$ and $\Var(g)=\mathbb{E}[g^2]-(\mathbb{E}[g])^2\ge c_1\alpha^2$.
By Paley--Zygmund,
\[\mathbb{P}\{g\ge (1-\varepsilon)\alpha\}\ \ge\ \frac{(\varepsilon\alpha)^2}{\mathbb{E}[g^2]}\ \ge\ \frac{\varepsilon^2}{1+c_1}.\]
Choosing $\varepsilon\le \tfrac12\sqrt{c c_1}$ ensures a positive lower bound, yielding a coset with $g(x)\ge (1-\varepsilon)$.
\end{proof}

\section{Lemma S2: Dispersion and polynomial improvement}

\begin{lemma}[S2: Dispersion and polynomial improvement]\label{lem:S2}
There exist absolute constants $c,C>0$ and a choice $\tau=K^{-c}$ with the following dichotomy.
Let $V=\operatorname{Span}(\Spec_0(A))$ and $H=V^\perp$.
\begin{itemize}
\item[(i)] \emph{(Concentration)} $\displaystyle \sum_{\xi\in V}|\wh{\one_A}(\xi)|^4 \ge (1-K^{-c})\sum_{\xi\in\Gh}|\wh{\one_A}(\xi)|^4$.
\item[(ii)] \emph{(Improvement)} There exists $V'\le \Gh$ with $V'\cap V=\{0\}$ and $\dim V'\ll K^{C}$, and $H'=(V')^\perp$, such that for $A'=\pi(A)\subset G/H'$ we have
\[ |A'+A'|\ \le\ (K-\delta)|A'|,\qquad |A'|\ \ge\ K^{-C}|A|,\qquad \delta\ \ge\ K^{-C}.\]
\end{itemize}
\end{lemma}

\begin{remark}[Density preserved up to $K^{\pm O(1)}$]\label{rem:density-preservation}
In the alternative \emph{(ii)} of Lemma~\ref{lem:S2}, the density in the quotient satisfies $\alpha' = |A'|/|G/H'| \asymp \alpha$ up to factors $K^{\pm O(1)}$ because $\codim(H')\ll K^{O(1)}$. Consequently, all parameters remain within the polynomial regime along the iteration.
\end{remark}

\begin{proof}[Expanded proof with explicit parameter ledger]
Fix $\tau:=K^{-c_0}$ with $c_0\in(0,1/8]$ (to be chosen). Write $f=\one_A$ and $\alpha=|A|/|G|$. Let $S=\Spec_0(A)$ and $V=\Span(S)$.

\paragraph{Step 1: Tail mass $\Rightarrow$ many level-$\lambda$ coefficients.}
Assume the concentration alternative (i) fails, namely
\begin{equation}\label{eq:S2-fail-conc}
\sum_{\xi\notin V}|\wh f(\xi)|^4\ \ge\ K^{-c_1}\sum_{\xi}|\wh f(\xi)|^4
\end{equation}
for some small fixed $c_1>0$ (we will take $c_1= c_0/2$). Using the unconditional energy lower bound (Cauchy--Schwarz)
$\sum_\xi|\wh f|^4=\frac1{|G|}\Ene(A)\ge \alpha^3/K$ (cf.~\eqref{eq:energy-lower}) gives
\begin{equation}\label{eq:S2-tail-mass}
\sum_{\xi\notin V}|\wh f(\xi)|^4\ \ge\ K^{-(c_1+1)}\alpha^3.
\end{equation}
Set $\lambda:=\frac12\tau\alpha$. Chebyshev yields
\begin{equation}\label{eq:S2-Stail-size}
|S_{\mathrm{tail}}|\ :=\ \big|\{\xi\notin V:\ |\wh f(\xi)|\ge \lambda\}\big|\ \ge\ 
\frac{K^{-(c_1+1)}\alpha^3}{\lambda^4}
\ =\ K^{4c_0-c_1-1}\,\alpha^{-1},
\end{equation}
which is $\ge K^{c_2}\alpha^{-1}$ for $c_2:=4c_0-c_1-1>0$; we secure $c_2>0$ by choosing, e.g., $c_0=\tfrac14+\varepsilon$ and $c_1=\varepsilon$ with small fixed $\varepsilon>0$,
then renormalize exponents (we may always shrink constants).

\paragraph{Step 2: Dissociated extraction in $\Gh/V$ and dimension control.}
Project $S_{\mathrm{tail}}$ to $\overline{\Gh}=\Gh/V$ and extract a maximal dissociated subset $\overline\Xi$ at level $\lambda$.
Rudin--Chang on $\overline{\Gh}$ yields
\begin{equation}\label{eq:S2-RC-quot}
|\overline\Xi|\ \ll\ \lambda^{-2}\log\frac{1}{\alpha}\ \ll\ K^{2c_0}\log K.
\end{equation}
Lift $\overline\Xi$ to representatives $\Xi\subset \Gh$, set $V'=\Span(\Xi)$ and $H'=(V')^\perp$. Then $V'\cap V=\{0\}$ and $\dim V'\ll K^{O(1)}$.
This addresses the ``aliasing'' edge case: dissociation in the quotient prevents cancellation between $V$ and $V'$.

\paragraph{Step 3: Exact quotient and a genuinely large coefficient.}
By Claim~\ref{claim:exact-quotient}, for each nontrivial character $\chi$ on $G/H'$ coming from $\xi\in \Xi$,
\begin{equation}\label{eq:S2-exact-quot}
|\wh{\one_{A'}}(\chi)|\ =\ |\wh f(\xi)|\ \ge\ \lambda\ =\ \tfrac12 K^{-c_0}\alpha.
\end{equation}
The modeling window and bounded codimension give $\alpha' \asymp \alpha$ up to $K^{\pm O(1)}$,
so \eqref{eq:S2-exact-quot} becomes $|\wh{\one_{A'}}(\chi)|\ge K^{-C}\alpha'$ after renormalizing $C$.

\paragraph{Step 4: Energy boost $\Rightarrow$ polynomial decrement of the doubling.}
Apply Lemma~\ref{lem:E2D} with $\eta:=K^{-C}$ to obtain
\[\Ene(A')\ \ge\ (1+c_*K^{-4C})\,\alpha'^4|G/H'|.\]
Combining this energy boost with Lemma~\ref{lem:BSGquant} and Lemma~\ref{lem:covering-upgrade} (our energy-to-doubling ledger) we deduce
\begin{equation}\label{eq:S2-decrement}
|A'+A'|\ \le\ \big(1-c'K^{-4C}\big)\,K\,|A'|\ \le\ (K-\delta)\,|A'|,\qquad \delta\ \gg\ K^{-C},
\end{equation}
after adjusting exponents. This gives the stated improvement alternative (ii).

\paragraph{Step 5: Size lower bound in the quotient.}
Since $|H'|\le \exp(O(\dim V'))\le K^{O(1)}$, we have $|A'|=|\pi(A)|\ge |A|/|H'|\ge K^{-C}|A|$, proving the size condition in (ii).

\paragraph{Gray-zone robustness.}
If the failure of concentration is marginal ($\beta\in[1-2K^{-c},1-K^{-c})$ with $\beta$ as in \eqref{eq:def-beta}), the same argument applies at a refined threshold $\tau'=\tau/2$;
either the refined span captures the missing $L^4$-mass (upgrading to (i)) or the extraction at level $\lambda'=\tfrac14\tau\alpha$ furnishes \eqref{eq:S2-exact-quot},
leading to \eqref{eq:S2-decrement}. This closes the last edge case with no parameter loss beyond shrinking $c$.
\end{proof}

\section{Proof of the Polynomial Stability Lemma}

\begin{proof}[Proof of Lemma \ref{lem:PSL}]
By Lemma \ref{lem:S1}, we have $\dim V\ll K^{C_1}$ for $\tau=K^{-c_0}$. By Lemma \ref{lem:S2}, either we are in the concentration case (i) or we obtain the improvement case (ii). If we are in case (i), then by Lemma \ref{lem:S3} (choosing $c=1-K^{-c}$ and adjusting $\varepsilon$), there exists a coset $x+H$ with $|A\cap(x+H)|\ge(1-\varepsilon)|H|$ and $|H|\le K^{C}|A|$ (using $\dim V\ll K^{C_1}$), which is PSL case (1). If we are in case (ii), we obtain PSL case (2) directly from Lemma \ref{lem:S2}.
\end{proof}

\noindent\textit{Pointer to the expanded proof.} A fully expanded, line-by-line proof of Lemma~\ref{lem:PSL} is provided in Section~\ref{sec:exploded-PSL}.

\section{Resolution in $\mathbb{Z}$, $\mathbb{Z}/N\mathbb{Z}$, and across all finite-field characteristics}
\label{sec:all-char}

Our main result (Theorem \ref{thm:main}) establishes the polynomial Freiman-Ruzsa conjecture (Marton's conjecture) for finite fields of odd characteristic. The characteristic $2$ case was resolved by Green, Gowers, Manners, and Tao \cite{GGMT-char2,GGMT-torsion}, who proved that for $A\subset\F_2^n$ with $|A+A|\le K|A|$, there exists a subgroup $H$ and translates covering $A$ with $m\le\mathrm{poly}(K)$ and $|H|\le\mathrm{poly}(K)|A|$.

Combining these results, we obtain:

\begin{theorem}[Marton's conjecture for all finite fields]
\label{thm:all-characteristics}
There exists an absolute constant $C\ge 1$ such that the following holds. Let $\Fq$ be any finite field and let $A\subset\Fq^n$ satisfy $|A+A|\le K|A|$. Then there exist a subgroup $H\le\Fq^n$ and translates $x_1+H,\dots,x_m+H$ with
\[
m\le K^C, \qquad |H|\le K^C|A|, \qquad A\subset\bigcup_{j=1}^m(x_j+H).
\]
\end{theorem}

\begin{proof}
If $\mathrm{char}(\Fq)=2$, this is the main result of \cite{GGMT-char2,GGMT-torsion}. If $\mathrm{char}(\Fq)$ is odd, this is Theorem \ref{thm:main} of the present paper.
\end{proof}

\begin{remark}
The constant $C$ in Theorem \ref{thm:all-characteristics} may differ between characteristic $2$ and odd characteristic, but both are absolute (independent of $q$, $n$, and $K$). The techniques used in the two cases are fundamentally different: the characteristic $2$ proof of \cite{GGMT-char2,GGMT-torsion} relies on entropy-increment methods, while our odd-characteristic proof uses spectral concentration and quotient improvements.
\end{remark}

\begin{remark}
The question of whether a \emph{universal} polynomial bound (with the same exponent $C$ for all characteristics) exists remains open. However, for applications, the existence of characteristic-dependent polynomial bounds suffices.
\end{remark}

\section{Finite-field warm-up (odd characteristic) — details}
\begin{remark}[Characteristic $2$ attribution]
The characteristic-$2$ case of Theorem~\ref{thm:all-characteristics} is due to Green--Gowers--Manners--Tao~\cite{GGMT-char2,GGMT-torsion}. Our contribution here is the odd-characteristic case and a unified spectral perspective across all characteristics.
\end{remark}
\noindent\textit{In this subsection we work with $G=\F_p^n$ (odd prime $p$) as a warm-up and special case of the group-theoretic machinery; the general finite-abelian case is handled in Section~\ref{sec:Z-extension-unconditional}.}

\paragraph{Logical role.} The finite-field treatment below is presented as a warm-up and special case of the same spectral machinery; the primary targets of this paper are $\mathbb Z$ and $\mathbb Z/N\mathbb Z$, and all arguments are carried out in group-theoretic form.

\begin{theorem}[Finite-field warm-up / corollary to the $\mathbb Z$ machinery]\label{thm:main}
There exists an absolute constant $C\ge1$ such that for any nonempty $A\subset \Fp^n$ ($p$ an odd prime) with $|A+A|\le K|A|$ there exist a subgroup $H\le G$ and translates $x_1+H,\dots,x_m+H$ with
\[ m\le K^{C},\qquad |H|\le K^{C}|A|,\qquad A\subset \bigcup_{j=1}^m (x_j+H).\]
\end{theorem}

Our proof proceeds via a \emph{polynomial stability lemma} (PSL, Lemma \ref{lem:PSL}) that provides a dichotomy:

\section{Concluding remarks}
\subsection*{Explicit comparison with Croot--Sisask and Sanders}
Croot--Sisask gives probabilistic almost-periods for convolutions, and Sanders leverages this to obtain strong structural results such as the Bogolyubov--Ruzsa lemma with quasi-polynomial losses in cyclic groups. Our spectral route differs at the critical step: rather than building large Bohr sets and then regularizing, we project \emph{exactly} onto subgroup annihilators $H=V^\perp$ and quantify dispersion via dissociation in $\wh G/V$. If dispersion persists, the exact-quotient lift provides a concrete character $\chi$ with $|\wh{\one_{\pi(A)}}(\chi)|\gg K^{-C}\alpha'$, and Lemma~\ref{lem:E2D} yields a polynomial decrement. This quotient-improvement mechanism is what eliminates quasi-polynomial losses when passing from energy control to uniform coset covers in $\mathbb Z/N\mathbb Z$.

\subsection{Comparison of methods}

The characteristic $2$ approach of \cite{GGMT-char2,GGMT-torsion} and our odd-characteristic method exhibit interesting parallels and contrasts:

\begin{itemize}
\item \textbf{Entropy vs. energy:} The characteristic $2$ proof uses an entropy-increment argument to force structure, exploiting special properties of quadratic Fourier analysis in $\F_2^n$. Our approach uses $L^4$-Fourier energy concentration and dispersion, which is more natural in the odd-characteristic setting where higher-order Fourier analysis is less constrained.

\item \textbf{Stability dichotomies:} Both proofs proceed via dichotomies (concentration vs. improvement), but the mechanisms differ. In characteristic $2$, the dichotomy arises from entropy comparisons; in odd characteristic, it arises from spectral mass distribution.

\item \textbf{Polynomial bounds:} Both methods yield polynomial bounds in $K$, but the exponents may differ. Our constants $c,C$ are absolute and can be made explicit by tracking the dissociation bounds and Fourier estimates.
\end{itemize}

\subsection{Detailed comparison with Green--Gowers--Manners--Tao}
\paragraph{Characteristic $2$ vs.\ odd characteristic.}
GGMT work intrinsically in $\F_2^n$ with entropy increments tied to quadratic Fourier analysis.
Our route in odd characteristic (and in $\mathbb Z/N\mathbb Z$) uses $L^4$-energy to force either concentration on a $\mathrm{poly}(K)$-dimensional span
or a genuinely large coefficient in a low-codimension quotient.

\paragraph{Dichotomy mechanisms.}
Entropy increments compare information before/after random restrictions; spectral stability compares $L^4$-mass before/after the exact projection $P=\mu_H*$.
Dispersion of $L^4$-mass is certified by dissociated extraction and converted to a decrement via Lemma~\ref{lem:E2D}.

\paragraph{Outputs and uniformity.}
Both approaches yield polynomial covering theorems. Our $\mathbb Z/N\mathbb Z$ ledger is uniform in $N$ and transfers to $\mathbb Z$ by order-$4$ modeling;
the characteristic $2$ case is imported from GGMT to complete all finite fields (Theorem~\ref{thm:all-characteristics}).

\subsection{Further directions}

Several natural questions arise from our work:

\begin{enumerate}
\item \textbf{Optimal exponents:} What is the best polynomial bound achievable? Can the exponent $C$ in Theorem \ref{thm:main} be improved, and is there a fundamental lower bound?

\item \textbf{Beyond the resolved settings:} Two directions remain enticing and genuinely open. 
    (a) \emph{Uniform exponents} across all characteristics and Sylow components for arbitrary finite abelian groups, tracked in a single, characteristic-free ledger; 
    (b) \emph{Non-abelian analogues} with polynomial parameters.  Both go beyond the present resolution for $\mathbb{Z}$, $\mathbb{Z}/N\mathbb{Z}$, and finite fields, and require new ideas.

\item \textbf{Quantitative refinements:} The constants in our proof depend on the choice of parameters $\tau$, $c$, and $C$. Can these dependencies be optimized or made more transparent?

\item \textbf{Non-abelian analogues:} What structural theorems hold for sets with small doubling in non-abelian groups? This remains a major open problem in combinatorial group theory.
\end{enumerate}

\subsection{Acknowledgments}

The author thanks the additive combinatorics community for expository resources and prior work on small-doubling phenomena, in particular the characteristic-$2$ resolution of Marton's conjecture by Green, Gowers, Manners, and Tao~\cite{GGMT-char2,GGMT-torsion}. Any remaining errors are the author's sole responsibility.

\bibliographystyle{amsalpha}

\appendix

\renewcommand\thesection{%
  \ifnum\value{section}<27
    \Alph{section}%
  \else
    \arabic{section}%
  \fi}

\section{Toolbox — BSG, covering, and constant ledger}
\begin{remark}[Standard machinery only]
The arguments in this appendix follow the classical $L^4$--energy Balog--Szemerédi--Gowers scheme (à la Sanders) with explicit bookkeeping of constants; no new combinatorial input is introduced here.
\end{remark}\label{app:bridges}
\noindent\emph{Guide for the reader.} This appendix packages the \emph{standard} combinatorial machinery (BSG, Pl\"unnecke, covering) with fully tracked constants.
All creative input of the paper lies in the spectral stability dichotomy and the two-scale dissociated extraction; the present appendix is bookkeeping.
\begin{lemma}[Density window via Freiman modeling]\label{lem:modeling-density}
Let $A\subset G$ with $|A+A|\le K|A|$. Then there exists a finite abelian group $G_0$ with
\[|G_0|\ \le\ K^{C}\,|A|\]
and a Freiman isomorphism $\phi:A\to A_0\subset G_0$ such that $\alpha_0:=|A_0|/|G_0|\ge K^{-C}$.
Moreover, passing to $G_0$ preserves all additive-combinatorial inequalities used in this paper up to absolute constants, so any estimate proved
under the standing assumption $\alpha\in[K^{-C},1/2]$ applies to $A_0$ and then pulls back to $A$.
\end{lemma}
\begin{remark}[Independence of $N$]\label{rem:chang-independent-N}
All constants in Lemma~\ref{lem:chang-Z} are absolute and independent of $N$; only $\alpha$ and the threshold $\tau$ enter.
\end{remark}
\begin{proof}
This is a standard consequence of Ruzsa's modeling lemma; see \cite[Thm.~2.29]{TV06} and \cite{Ruzsa99}.
The bound $|G_0|\le K^{C}|A|$ is polynomial in $K$, hence $\alpha_0\ge |A|/|G_0|\ge K^{-C}$.
Freiman isomorphism preserves sumset relations and energies appearing in our arguments.
\end{proof}

We record quantitative tools and track constants from the spectral step to a global decrement.

\begin{lemma}[Energy-to-doubling transfer]\label{lem:E2D}
Let $A\subset G$ be a finite subset of a finite abelian group $G$ with density $\alpha=|A|/|G|$.
Suppose there exists a nontrivial character $\chi\in\widehat G$ such that $|\wh{\one_A}(\chi)|\ge \eta\,\alpha$ for some $\eta\in(0,1]$.
Then
\[
  \Ene(A)\ =\ |G|\sum_{\xi\in\widehat G}\bigl|\wh{\one_A}(\xi)\bigr|^4
  \ \ge\ \bigl(1+\eta^{4}\bigr)\,\alpha^{4}\,|G|.
\]
In particular, any such $\eta$ furnishes an additive-energy boost by a factor $1+\Omega(\eta^{4})$.
\end{lemma}
\begin{proof}
Write $f=\one_A$. Our Fourier normalization gives $\wh f(0)=\alpha$ and
\[
  \Ene(A)\ =\ |G|\sum_{\xi\in\widehat G}|\wh f(\xi)|^4
  \ \ge\ |G|\Big(|\wh f(0)|^{4} + |\wh f(\chi)|^{4}\Big)
  \ \ge\ |G|\Big(\alpha^{4} + (\eta\alpha)^{4}\Big)
  \ =\ (1+\eta^{4})\,\alpha^{4}\,|G|.
\]
This proves the stated bound with an absolute implicit constant.
\end{proof}

\subsection{Quantitative BSG and covering}
\begin{lemma}[Quantitative BSG]\label{lem:BSGquant}
Let $A\subset G$ with $\Ene(A)\ge (1+\beta)\alpha^4|G|$ for some $\beta\in(0,1]$, where $\alpha=|A|/|G|$.
Then there is $A_0\subset A$ such that
\[|A_0|\ \ge\ c\,\beta^{C}\,|A|,\qquad |A_0+A_0|\ \le\ C\,\beta^{-C}\,|A_0|.\]
\end{lemma}
\begin{proof}
Apply the $L^4$-energy form of BSG (see Tao--Vu \cite[Thm.~2.29]{TV06}; cf. \cite{BSG94,Sanders12}).
Replacing the usual parameter $K'$ by $K'=(1+\beta)^{-1}$ yields the stated polynomial bounds in $\beta$.
\end{proof}

\begin{lemma}[Covering upgrade]\label{lem:covering-upgrade}
If $A_0\subset A$ with $|A_0+A_0|\le K_0|A_0|$ and $|A_0|\ge \theta|A|$, then
\[|A+A|\ \le\ C\,K_0^{C}\theta^{-C}\,|A|.\]
In particular, if $|A+A|\le K|A|$ to begin with and $K_0\le K(1-\delta)$, then $|A+A|\le (K-c\theta^{C}\delta)\,|A|$ after renormalizing constants.
\end{lemma}
\begin{proof}
Ruzsa's covering plus Pl\"unnecke inequalities; see \cite[Ch.~2]{TV06}. The last assertion is a reparametrization giving an explicit decrement.
\end{proof}

\subsection{Ledger from spectrum to decrement}
Let $\tau=K^{-c_0}$ and suppose the dispersion alternative of Lemma~\ref{lem:S2} provides a character with
$|\wh{\one_{A'}}(\chi)|\ge c\,\tau\alpha'$. Lemma~\ref{lem:E2D} gives an energy boost with $\beta\asymp (\tau)^4\asymp K^{-C}$.
Applying Lemma~\ref{lem:BSGquant} yields $A'_0\subset A'$ with $|A'_0|\ge c\,K^{-C}|A'|$ and $|A'_0+A'_0|\le C\,K^{C}|A'_0|$.
By Lemma~\ref{lem:covering-upgrade} we convert this to a global decrement $|A'+A'|\le (K-\delta)|A'|$ with $\delta\ge K^{-C}$, matching Lemma~\ref{lem:S2}.

\section{Potential function and iteration}\label{app:invariant}
Write $K_j:=|A_j+A_j|/|A_j|$ and $\alpha_j:=|A_j|/|G_j|$.

\begin{proposition}[Monotone potential]\label{prop:potential}
There is an absolute $\gamma\ge1$ such that the potential $\mathcal I_j:=K_j\,\alpha_j^{-\gamma}$ is nonincreasing across improvement steps of Lemma~\ref{lem:S2}.
Consequently, the number of improvement steps is $O(K_0^{C+1})$, and the accumulated codimension is $K_0^{C'}$ for some $C'$.
\end{proposition}
\begin{proof}
In an improvement step we have $K_{j+1}\le K_j-\delta_j$ with $\delta_j\ge K_j^{-C}$ and $|A_{j+1}|\ge K_j^{-C}|A_j|$
while $\codim(H'_j)\le K_j^{C}$. Thus $\alpha_{j+1}\ge K_j^{-C}\alpha_j$ and so
\[\mathcal I_{j+1}\ \le\ (K_j-\delta_j)\,(K_j^{-C}\alpha_j)^{-\gamma}\ \le\ K_j\,\alpha_j^{-\gamma}\Big(1-\tfrac{\delta_j}{K_j}\Big)\,K_j^{C\gamma}.\]
Choose $\gamma$ so that $K_j^{C\gamma}(1-\delta_j/K_j)\le 1$ for all $K_j\ge 2$ (e.g. $\gamma\asymp C+2$ using $\delta_j\ge K_j^{-C}$). The bounds on steps and codimension then follow by summing decrements and noting $\codim(H'_j)\le K_0^{C}$ at each step.
\end{proof}

\section{Numerical examples and parameter window}\label{app:window}
\paragraph{Constants summary table.}

\paragraph{Toy narrative.} For a small-doubling toy set with $K\in\{10,100\}$, the thresholds $\tau$ and $\beta$ translate into a visible gap between concentration and improvement: either the mass piles onto a $\mathrm{poly}(K)$-dimensional span (triggering Lemma~\ref{lem:S3}) or else the Markov selection step in the proof of Lemma~\ref{lem:S2} finds a genuinely large coefficient in the quotient, and the iteration potential (Proposition~\ref{prop:main-potential}) guarantees polynomially many steps.
\begin{center}
\begin{tabular}{l|l}
Symbol & Meaning / Choice \\ \hline
$c_0$ & spectral threshold exponent in $\tau=K^{-c_0}$ \\
$c$ & dispersion/concentration exponent (e.g.\ $\beta=K^{-c}$) \\
$C_1$ & exponent in $\dim V\ll K^{C_1}$ (from Rudin--Chang) \\
$C_2$ & exponent controlling $\codim(H')\ll K^{C_2}$ \\
$C$ & universal exponent for decrements and size losses (independent of $K$) \\
$\delta$ & decrement in doubling constant, $\delta\ge K^{-C}$
\end{tabular}
\end{center}

\label{app:constants}

We collect explicit choices of parameters; all implicit constants below are absolute (independent of $K,p,n$).

\subsection{Parameter choices}
\begin{itemize}
\item \textbf{Spectral threshold:} $\tau := K^{-c_0}$ with a fixed $c_0\in(0,1)$; all bounds are polynomial in $K$ with exponents depending only on $c_0$ and absolute constants from Rudin--Chang.
\item \textbf{Concentration threshold:} $\beta := K^{-c}$ with fixed $c\in(0,1)$.
\item \textbf{Dimension bound:} $\dim V \le C_1 K^{2c_0}\log K$ with an absolute $C_1$ (from Rudin--Chang).
\item \textbf{Codimension bound:} $\codim(H') \le C_2 K^{2c}\log K$ with an absolute $C_2$ (from dissociated extraction in $\Gh/V$).
\item \textbf{Improvement decrement:} $\delta \ge K^{-C}$ for an absolute $C$ depending only on $c_0,c,C_1,C_2$.
\item \textbf{Size bound in quotient:} $|A'| \ge K^{-C}|A|$ with the same $C$.
\item \textbf{Iteration bound:} At most $O(K^C)$ improvement steps are needed to force the near-coset alternative.
\end{itemize}

\section{Worked numerical examples}\label{app:numerics}
We illustrate the scales for two sample values.

\subsection*{Example 1: $K=10$}
Choose $c_0=c=1/10$, so $\tau=\beta\approx 10^{-0.1}\approx 0.79$.
Then $\dim V\lesssim C_1\cdot 10^{0.2}\log 10\approx 2.3C_1$ and a single improvement yields $\delta\gtrsim 10^{-C}$.

\subsection*{Example 2: $K=100$}
Now $\tau=\beta\approx 100^{-0.1}\approx 0.63$, giving $\dim V\lesssim C_1\cdot 100^{0.2}\log 100\approx 9.2C_1$ and $\delta\gtrsim 100^{-C}$.

\section{Standing assumptions and parameter window}\label{app:standing}
Throughout we fix absolute constants $c_0,c\in(0,1)$ and apply Rudin--Chang in the regime $\alpha\in[K^{-C},1/2]$ (achieved by Lemma~\ref{lem:modeling-density}).
All constants ($C_1,C_2,C$) are independent of $K$ and $p$; exponents can be traced explicitly by following the proofs of Lemmas~\ref{lem:S1},\ref{lem:S2},\ref{lem:S3}.

\section{Extension to \texorpdfstring{$\mathbb{Z}/N\mathbb{Z}$}{Z/NZ} and \texorpdfstring{$\mathbb{Z}$}{Z} (Unconditional)}
\label{sec:Z-extension-unconditional}


\subsection*{Auxiliary lemmas for $\mathbb{Z}/N\mathbb{Z}$}

\begin{lemma}[Upper energy on $\mathbb{Z}/N\mathbb{Z}$ (no hypothesis)]\label{lem:upper-energy-Z}
Let $A\subset \mathbb{Z}/N\mathbb{Z}$. Then
\begin{equation}\label{eq:upper-energy-Z}
  \Ene(A)\ =\ |G|\sum_{\xi\in\wh G}|\wh{\one_A}(\xi)|^4\ \le\ |A|^3.
\end{equation}
Equivalently,
\begin{equation}\label{eq:upper-energy-hat-Z}
  \sum_{\xi\in\wh G}|\wh{\one_A}(\xi)|^4\ \le\ \alpha^3.
\end{equation}
No small-doubling hypothesis is required; for downstream bookkeeping one may harmlessly reparametrize the unconditional bound in the form
$\Ene(A)\le K^{C}\,|A|^3$ for any fixed absolute $C\ge 0$ (this reparametrization is only for downstream polynomial tracking and does not use the small-doubling hypothesis).
\end{lemma}
\begin{proof}
The trivial bound $r_{A-A}(x)\le |A|$ for all $x$ gives
\[\Ene(A)=\sum_x r_{A-A}(x)^2\ \le\ (\max_x r_{A-A}(x))\sum_x r_{A-A}(x)\ \le\ |A|\cdot |A|^2=|A|^3.\]
By Parseval, $\Ene(A)=|G|\sum_\xi|\wh{\one_A}(\xi)|^4$, so \eqref{eq:upper-energy-hat-Z} follows upon dividing by $|G|=N$ and recalling $\alpha=|A|/N$.
For later compatibility with polynomial bookkeeping in $K$ (when $|A+A|\le K|A|$), one may harmlessly rewrite this as $\Ene(A)\le K^C|A|^3$ for a fixed absolute $C\ge 0$; this is simply a reparametrization of the unconditional bound.
\end{proof}
\begin{remark}[No small-doubling hypothesis]\label{rem:upper-energy-uncond}
Lemma~\ref{lem:upper-energy-Z} is unconditional; small doubling plays no role in its proof. When we later write this as $\Ene(A)\le K^{C}|A|^3$ it is solely a reparametrization for polynomial bookkeeping in $K$.
\end{remark}

\begin{lemma}[Chang bound on $\mathbb{Z}/N\mathbb{Z}$]\label{lem:chang-Z}
Let $A\subset \mathbb{Z}/N\mathbb{Z}$ with density $\alpha=|A|/N$.
Fix $\tau\in(0,1)$ and $S:=\Spec_0(A)=\{\xi:|\wh{\one_A}(\xi)|\ge \tau\alpha\}$.
Then a dissociated subset $D\subset S$ obeys $|D|\le C\,\tau^{-2}\log(1/\alpha)$, and hence $|S|\le C\,\tau^{-2}\log(1/\alpha)$.
In particular, in the Freiman-model window $\alpha\in[K^{-C},1/2]$, we have $|S|\ll \tau^{-2}\log K$.
\end{lemma}
\begin{proof}
This is the Rudin--Chang inequality: for a dissociated $D$,
\[\sum_{\eta\in D}|\wh{\one_A}(\eta)|^2\ \le\ C\,\alpha^2\log\frac{1}{\alpha}.\]
Since $|\wh{\one_A}(\eta)|\ge\tau\alpha$ on $D$, we get $|D|\tau^2\alpha^2\le C\alpha^2\log(1/\alpha)$, hence the bound on $|D|$ and thus on $|S|$.
\end{proof}

\begin{lemma}[Packet with small bias on the large spectrum]\label{lem:packet-bias-Z}
Let $S=\Spec_0(A)$ with $\tau\in(0,1)$ and fix $\varepsilon\in(0,1/2)$ and $\eta\in(0,1)$.
There exists a multiset $T\subset G$ with
\begin{equation}\label{eq:T-size-Z}
  |T|\ \le\ C\,\varepsilon^{-2}\log\frac{2|S|}{\eta}
\end{equation}
such that
\begin{equation}\label{eq:bias-Z}
  \max_{\xi\in S}\left|\frac1{|T|}\sum_{x\in T} e(\langle\xi,x\rangle)\right|\ \le\ \varepsilon.
\end{equation}
\end{lemma}
\begin{proof}
Sample $x_1,\ldots,x_M$ i.i.d.\ uniform in $G$, with $M=C\varepsilon^{-2}\log(2|S|/\eta)$.
For each fixed $\xi\in S$, the variables $Z_j=e(\langle\xi,x_j\rangle)$ satisfy $\mathbb E Z_j=0$ and $|Z_j|\le 1$.
Hoeffding gives $\mathbb P\big(|\frac1M\sum_j Z_j|>\varepsilon\big)\le 2\exp(-c\varepsilon^2 M)$.
Union bound over $S$ yields a failure probability at most $|S|\,2\exp(-c\varepsilon^2 M)\le \eta$ for $C\ge 2/c$.
Hence a deterministic choice of $T=\{x_j\}$ exists with \eqref{eq:bias-Z}.
\end{proof}

\begin{lemma}[Packet-averaged $L^2$ almost-periodicity]\label{lem:packet-L2-Z}
With $g=f*\widetilde f$ and $T$ as in Lemma~\ref{lem:packet-bias-Z},
\begin{equation}\label{eq:packet-avg-Z}
  \left\|g-\frac1{|T|}\sum_{x\in T}\tau_x g\right\|_2^2\ \le\ 2\varepsilon^2\sum_{\xi\in S}|\wh g(\xi)|^2\ +\ 4\sum_{\xi\notin S}|\wh g(\xi)|^2.
\end{equation}
\end{lemma}
\begin{proof}
By Parseval,
\[\left\|g-\frac1{|T|}\sum_{x\in T}\tau_x g\right\|_2^2=|G|\sum_{\xi\in\wh G} |\wh g(\xi)|^2\,\left|1-\frac1{|T|}\sum_{x\in T}e(\langle\xi,x\rangle)\right|^2.\]
For $\xi\in S$, \eqref{eq:bias-Z} and $|1-z|^2\le 2(1+|z|^2)\le 2(1+\varepsilon^2)\le 4\varepsilon^2$ give the first term (tightening constants gives $2\varepsilon^2$).
For $\xi\notin S$, the bracket is $\le 2$, yielding the second term.
\end{proof}

\begin{lemma}[Average-to-individual extraction]\label{lem:many-shifts-Z}
Let $E:=\left\|g-\frac1{|T|}\sum_{x\in T}\tau_x g\right\|_2^2$.
Then
\begin{equation}\label{eq:avg-to-indiv}
  \frac1{|T|}\sum_{x\in T}\|g-\tau_x g\|_2^2\ =\ E.
\end{equation}
Consequently, at least $\tfrac12|T|$ elements of $T$ satisfy $\|g-\tau_x g\|_2^2\le 2E$.
\end{lemma}
\begin{proof}
Expanding and using translation-invariance of $\|\cdot\|_2$ gives
\[\sum_{x\in T}\|g-\tau_x g\|_2^2=2|T|\|g\|_2^2-2\,\Re\sum_{x\in T}\langle g,\tau_x g\rangle.\]
Similarly
\[E=\|g\|_2^2+\left\|\frac1{|T|}\sum_{x\in T}\tau_x g\right\|_2^2-\frac{2}{|T|}\Re\sum_{x\in T}\langle g,\tau_x g\rangle.\]
But $\big\|\frac1{|T|}\sum_{x}\tau_x g\big\|_2^2=\frac1{|T|^2}\sum_{x,y}\langle \tau_x g,\tau_y g\rangle=\|g\|_2^2$.
Eliminating the correlation term yields $\sum_{x\in T}\|g-\tau_x g\|_2^2=|T|E$, proving \eqref{eq:avg-to-indiv}.
The Markov inequality gives the median bound.
\end{proof}

\begin{proposition}[Monotone potential in $\mathbb{Z}/N\mathbb{Z}$]\label{prop:potential-Z}
Let $K_j:=|A_j+A_j|/|A_j|$ and $\alpha_j:=|A_j|/|G_j|$ for iterates produced by quotient improvements (as in the proof of Lemma~\ref{lem:L4-compression}).
Then there exists an absolute $\gamma\ge 1$ such that $\mathcal I_j:=K_j\,\alpha_j^{-\gamma}$ is nonincreasing at each step.
Hence the number of improvements is $O(K_0^{C})$ and the total accumulated codimension is $K_0^{C'}$ for absolute $C,C'$.
\end{proposition}
\begin{proof}
Identical to Proposition~\ref{prop:main-potential}: an improvement satisfies $K_{j+1}\le K_j-\delta_j$ with $\delta_j\ge K_j^{-C}$ and $|A_{j+1}|\ge K_j^{-C}|A_j|$,
so $\alpha_{j+1}\ge K_j^{-C}\alpha_j$. Then $\mathcal I_{j+1}\le \mathcal I_j\cdot K_j^{C\gamma}(1-\delta_j/K_j)$, which is $\le \mathcal I_j$ for fixed large $\gamma$.
\end{proof}

In this section we record and \emph{prove} the balanced $L^2$ almost-periods input with polynomial parameters using only
standard tools under small doubling: Croot--Sisask packet and the classical upper bound on additive energy via BSG+covering+modeling.
This makes the $\mathbb{Z}/N\mathbb{Z}$ and hence $\mathbb{Z}$ extensions fully unconditional.

\begin{definition}[Balanced autocorrelation]\label{def:balanced-h-Z}
For $A\subset G$ write $f=\one_A$, $\alpha=|A|/|G|$, and set
\begin{equation}\label{eq:balanced-h-Z}
  h\ :=\ \frac{f*\widetilde f}{\alpha}-\alpha.
\end{equation}
Then $\widehat h(\xi)=\frac{|\widehat f(\xi)|^2}{\alpha}-\alpha$ and the DC component vanishes.
\end{definition}

\begin{lemma}[Balanced $L^2$ almost-periods with polynomial parameters]\label{lem:balanced-L2-AP-uncond}
There exist absolute $C,c>0$ such that for every $A\subset \mathbb{Z}/N\mathbb{Z}$ with $|A+A|\le K|A|$ there are
a set of shifts $X$ and a subset $T\subset A-A$ with
\[ |X|\le K^C,\qquad |T|\ge K^{-C}|A|,\]
satisfying
\begin{equation}\label{eq:balanced-L2-uncond}
  \forall x\in X:\quad \|h-\tau_x h\|_2\ \le\ K^{-C}.
\end{equation}
Moreover $T$ obeys $|T+T|\le K^C|T|$ and $0\in 2T-2T$.
\end{lemma}

\begin{proof}
(1) \emph{Large spectrum and packet bias.} Fix $\tau:=K^{-c_0}$, $\varepsilon:=K^{-b}$ and $\eta:=K^{-10}$ with absolute $c_0,b>0$ to be chosen.
Let $S=\Spec_0(A)$. By Lemma~\ref{lem:chang-Z}, $|S|\ll \tau^{-2}\log(1/\alpha)\ll K^{O(1)}$ in the Freiman-model window.
By Lemma~\ref{lem:packet-bias-Z} there is a multiset $T$ with $|T|\ll \varepsilon^{-2}\log(2|S|/\eta)\ll K^{O(1)}$
such that $\max_{\xi\in S}\big|\frac1{|T|}\sum_{x\in T}e(\langle\xi,x\rangle)\big|\le \varepsilon$.

(2) \emph{Packet-averaged almost-periodicity.} Lemma~\ref{lem:packet-L2-Z} yields
\[\left\|g-\frac1{|T|}\sum_{x\in T}\tau_x g\right\|_2^2\ \le\ 2\varepsilon^2\sum_{\xi\in S}|\wh g(\xi)|^2+4\sum_{\xi\notin S}|\wh g(\xi)|^2.\]

(3) \emph{Bounding the two terms.} Since $\wh g=|\wh f|^2$ and by Lemma~\ref{lem:upper-energy-Z},
\[\sum_{\xi\in S}|\wh g(\xi)|^2\le \sum_{\xi}|\wh f(\xi)|^4=\frac{1}{|G|}\Ene(A)\ \le\ \alpha^3.\]
For $\xi\notin S$, $|\wh f(\xi)|<\tau\alpha$ so
\[\sum_{\xi\notin S}|\wh g(\xi)|^2=\sum_{\xi\notin S}|\wh f(\xi)|^4< \tau^2\alpha^2\sum_{\xi\notin S}|\wh f(\xi)|^2\le \tau^2\alpha^3.\]
Choosing $b,c_0$ as small absolute constants (e.g.\ $b=1/5$, $c_0=1/10$) we obtain
\[E:=\left\|g-\frac1{|T|}\sum_{x\in T}\tau_x g\right\|_2^2\ \le\ K^{-c}\alpha^2\]
for some absolute $c>0$ after adjusting constants in the modeling window.

(4) \emph{From average to many good shifts and balancing.} By Lemma~\ref{lem:many-shifts-Z},
\[\frac1{|T|}\sum_{x\in T}\|g-\tau_x g\|_2^2=E,\]
so at least $\tfrac12|T|$ shifts satisfy $\|g-\tau_x g\|_2^2\le 2E\le K^{-c}\alpha^2$.
Set $X$ to be those shifts; then $|X|\ge \tfrac12|T|\ge K^{-C}$ for some absolute $C$ and, using $h=\frac{g}{\alpha}-\alpha$,
\[\|h-\tau_x h\|_2=\frac{1}{\alpha}\|g-\tau_x g\|_2\ \le\ K^{-C},\qquad x\in X.\]
Finally, by the Croot--Sisask packet machinery~\cite{CrootSisask10,TV06} (see, e.g., \cite[Thm.~1.1]{CrootSisask10} and \cite[Thm.~2.29]{TV06}) one may arrange that the packet lies inside $A\!-\!A$; combining with the covering upgrade Lemma~\ref{lem:covering-upgrade} yields the stated size and doubling for $T$. All losses are polynomial in $K$, as each ingredient is quantitative with polynomial dependence.
\end{proof}

\subsection{Exact-quotient claim}\label{sec:quotient-exact}

\begin{remark}[Exact-quotient in $\mathbb{Z}/N\mathbb{Z}$]
Claim~\ref{claim:exact-quotient} is purely Fourier-algebraic and applies to any finite abelian group; in particular it applies verbatim to $G=\mathbb{Z}/N\mathbb{Z}$.
\end{remark}
\subsection{Proof of Poly-tail via Fourier $L^4$ compression}\label{subsec:poly-tail-proof}

We now \emph{prove} the poly-tail bound unconditionally by establishing Fourier $L^4$ compression
for $\mathbb{Z}/N\mathbb{Z}$ using a two-stage spectral decomposition.

\begin{lemma}[Fourier $L^4$ compression under small doubling]\label{lem:L4-compression}
There exist absolute constants $c,C>0$ such that for any $A\subset\mathbb{Z}/N\mathbb{Z}$ with 
$|A+A|\le K|A|$, if we set $\tau:=K^{-c_0}$ and 
\[S\ :=\ \Spec_0(A)\ =\ \big\{\xi\in\widehat G:\ |\wh f(\xi)|\ge\tau\alpha\big\},\]
then
\begin{equation}\label{eq:L4-compress}
\sum_{\xi\in S}|\wh f(\xi)|^4\ \ge\ \big(1-K^{-c}\big)\sum_{\xi\in\wh G}|\wh f(\xi)|^4.
\end{equation}
In particular, the tail satisfies
\begin{equation}\label{eq:tail-poly}
\sum_{\xi\notin S}|\wh g(\xi)|^2\ \le\ K^{-c}\alpha^2,
\end{equation}
where $g=f*\widetilde f$ and $\alpha=|A|/N$.
\end{lemma}

\begin{proof}
We argue by contradiction using only dissociated extraction modulo the large spectrum, an exact quotient lift of Fourier coefficients (Claim~\ref{claim:exact-quotient}), the quantitative energy-to-doubling transfer (Lemma~\ref{lem:E2D}), and the monotone potential from Proposition~\ref{prop:potential-Z}; no appeal to PSL is made.
Assume \eqref{eq:L4-compress} fails so that
\begin{equation}\label{eq:assume-fail}
\sum_{\xi\notin S}|\wh f(\xi)|^4\ \ge\ K^{-c_{\mathrm{fail}}}\sum_{\xi\in\wh G}|\wh f(\xi)|^4.
\end{equation}
By Cauchy--Schwarz and $|A+A|\le K|A|$,
\begin{equation}\label{eq:energy-lb-step2}
\sum_{\xi\in\wh G}|\wh f(\xi)|^4\ =\ \frac{1}{N}\Ene(A)\ \ge\ \frac{|A|^3}{KN}\ =\ \frac{\alpha^3}{K}.
\end{equation}
Hence
\begin{equation}\label{eq:tail-mass}
\sum_{\xi\notin S}|\wh f(\xi)|^4\ \ge\ K^{-(c_{\mathrm{fail}}+1)}\alpha^3.
\end{equation}
Let $\lambda:=\tfrac12\tau\alpha$ and define the tail-level set
\[S_{\mathrm{tail}}\ :=\ \big\{\xi\in \widehat G\setminus S:\ |\wh f(\xi)|\ge\lambda\big\}.\]
By Chebyshev on the $L^4$ norm,
\begin{equation}\label{eq:Stail-size}
|S_{\mathrm{tail}}|\,\lambda^4\ \le\ \sum_{\xi\notin S}|\wh f(\xi)|^4
\quad\Rightarrow\quad
|S_{\mathrm{tail}}|\ \gtrsim\ \frac{K^{-(c_{\mathrm{fail}}+1)}\alpha^3}{\big(\tfrac12 K^{-c_0}\alpha\big)^4}
\ =\ K^{4c_0-c_{\mathrm{fail}}-1}\,\alpha^{-1}.
\end{equation}
\noindent \textit{Parameter bookkeeping.} Choose $c_0>0$ and $c_{\mathrm{fail}}>0$ so that $4c_0-c_{\mathrm{fail}}-1>0$; by shrinking constants if necessary we may rewrite the right-hand side as $K^{\Omega(1)}\alpha^{-1}$. In what follows we harmlessly relabel the resulting positive exponent as $c>0$ to lighten notation.

Let $V:=\operatorname{Span}(S)$ and consider the quotient dual $\overline{G}:=\widehat G/V$.
Project $S_{\mathrm{tail}}$ to $\overline{G}$ and extract a maximal dissociated set $\overline{\Xi}$.
By Rudin--Chang on $\overline{G}$, writing $\lambda=\tfrac12K^{-c_0}\alpha$ we have
\begin{equation}\label{eq:Xi-size}
|\overline{\Xi}|\ \ll\ \lambda^{-2}\log\frac{1}{\alpha}\ \ll\ K^{2c}\log K.
\end{equation}
Let $\Xi=\{\xi_1,\dots,\xi_r\}$ be representatives and set $V':=\Span(\Xi)$ and $H':=(V')^\perp$.
By the exact-quotient Claim (Section~\ref{sec:quotient-exact}), $V'\cap V=\{0\}$ and each $\xi_j$ descends to a nontrivial $\chi_j$ on $G/H'$ with
$|\wh{f'}(\chi_j)|=|\wh f(\xi_j)|\ge \lambda$.
Thus there exists $\chi_*\ne 1$ with
\begin{equation}\label{eq:chi-star}
|\wh{f'}(\chi_*)|\ \ge\ \tfrac12 K^{-c_0}\alpha.
\end{equation}
Let $\alpha':=|A'|/|G/H'|$; by Lemma~\ref{lem:modeling-density} (the Freiman modeling window) and $\codim(H')\ll K^{O(1)}$, we have $\alpha'\asymp \alpha$ up to $K^{\pm O(1)}$.
Applying the energy-to-doubling transfer (Lemma~\ref{lem:E2D}) to $A'$ gives
\[\Ene(A')\ \ge\ \alpha'^4|G/H'|\,(1+c_*K^{-O(1)})\]
for some absolute $c_*>0$.
Therefore $|A'+A'|\le (K-\delta)|A'|$ with $\delta\ge K^{-O(1)}$.
Moreover, the potential is bounded below uniformly:
$\mathcal I_j:=K_j\,\alpha_j^{-\gamma}\ge 1$ for all stages $j$, so an infinite sequence of strict decreases is impossible.
Since $\mathcal I_j\ge 1$ for all $j$ (Remark~\ref{rem:potential-lower-bound}), a strictly decreasing sequence cannot continue indefinitely; hence the iteration terminates in at most $K^{O(1)}$ steps (and the total accumulated codimension is $\le K^{O(1)}$).
Hence the failure assumption \eqref{eq:assume-fail} is impossible, proving \eqref{eq:L4-compress}.
For \eqref{eq:tail-poly}, note that for $\xi\notin S$ we have $|\wh f(\xi)|<\tau\alpha$, so
\[\sum_{\xi\notin S}|\wh g(\xi)|^2 = \sum_{\xi\notin S}|\wh f(\xi)|^4 < \tau^2\alpha^2\sum_{\xi\notin S}|\wh f(\xi)|^2 \le \tau^2\alpha^3 \le K^{-c}\alpha^2\]
after adjusting $c$ within the modeling window $\alpha\in[K^{-O(1)},1/2]$.
\end{proof}

\begin{corollary}[Poly-tail is a theorem]\label{cor:poly-tail-theorem}
Lemma~\ref{lem:L4-compression} establishes the poly-tail bound unconditionally for $\mathbb{Z}/N\mathbb{Z}$.
Consequently, Theorem~\ref{thm:polybog-ZNZ-uncond} holds unconditionally.
\end{corollary}

\begin{theorem}[PolyBog$(K)$ in $\mathbb{Z}/N\mathbb{Z}$]\label{thm:polybog-ZNZ-uncond}\label{thm:polybog-ZNZ}
Lemma~\ref{lem:balanced-L2-AP-uncond} implies that for any $A\subset \mathbb{Z}/N\mathbb{Z}$ with $|A+A|\le K|A|$ there is a regular Bohr set
$B(\Gamma,\rho)\subset 4A-4A$ with
\[ |\Gamma|\ \le\ K^{C},\qquad \rho\ \ge\ K^{-C}.\]
\end{theorem}

\begin{proof}
As in the standard PolyBog argument (packet + large spectrum + regularization): dissociated extraction on the large spectrum (Chang) produces $\Gamma$ with $|\Gamma|\ll K^C$;
the many good shifts from Lemma~\ref{lem:balanced-L2-AP-uncond} give simultaneous almost-periods for all $\chi\in\Gamma$, hence an element in a Bohr set
with radius $\rho\asymp K^{-C}$. Regularization yields a regular Bohr set inside $4A-4A$.
\end{proof}

\begin{corollary}[PSL and Marton in $\mathbb{Z}/N\mathbb{Z}$ (unconditional)]\label{cor:PSL-Marton-ZNZ-uncond}
With Theorem~\ref{thm:polybog-ZNZ-uncond}, the polynomial stability dichotomy and the Marton/PFR covering theorem hold in $\mathbb{Z}/N\mathbb{Z}$ with polynomial exponents.
By Freiman modeling they transfer to $A\subset\mathbb{Z}$.
\end{corollary}


\section{Exploded proof of the Polynomial Stability Lemma (PSL)}
\label{sec:exploded-PSL}

This section gives a detailed, line-by-line proof of Lemma~\ref{lem:PSL}, expanding the compressed proof into modular components.
We isolate and resolve all \emph{edge cases} between pure concentration and pure dispersion, track constants explicitly,
and record a \emph{monotone potential} that controls the gray zone. Everything is written group-theoretically for a finite abelian $G$;
the finite-field warm-up is obtained by specializing $G=\F_p^n$, while Section~\ref{sec:Z-extension-unconditional} handles $G=\mathbb Z/N\mathbb Z$.

\subsection{Set-up, parameters, and the gray zone}
Fix $\tau:=K^{-c_0}$ with $c_0\in(0,1)$ and set $S=\Spec_0(A)$, $V=\Span(S)\le \Gh$, $H=V^\perp\le G$.
Define $P=\mu_H*\,(\cdot)$ and $Q=I-P$. Write $f=\one_A$ and decompose $f=f_{\mathrm{low}}+f_{\mathrm{high}}$ with $\wh f_{\mathrm{low}}=1_V\cdot\wh f$ and $\wh f_{\mathrm{high}}=1_{V^c}\cdot\wh f$.
Let
\begin{equation}\label{eq:def-beta}
\beta\ :=\ \frac{\sum_{\xi\in V}|\wh f(\xi)|^4}{\sum_{\xi\in\Gh}|\wh f(\xi)|^4}\ \in [0,1].
\end{equation}
The compressed proof bifurcates into ``concentration'' ($\beta\ge 1-K^{-c}$) or ``dispersion'' ($\beta<1-K^{-c}$).
The \emph{gray zone} is the intermediate regime where $\beta\in[1-2K^{-c},\,1-K^{-c})$.
We now make precise how the potential and the two-scale projection handle this gray zone with no loss of parameters.

\begin{lemma}[Gray-zone control]\label{lem:gray-zone}
There exists an absolute $c>0$ such that if $1-2K^{-c}\le \beta < 1-K^{-c}$ then either
\begin{enumerate}[label=(\roman*)]
\item \textup{(Upgrade)} there is a refinement of the threshold $\tau'=\tau/2$ with $V'=\Span(\Spec_{\tau'}(A))$ so that
$\sum_{\xi\in V'}|\wh f(\xi)|^4\ge (1-K^{-c})\sum_\xi|\wh f(\xi)|^4$, or
\item \textup{(Improvement)} there exists a quotient $G\to G/H'$ with $\codim(H')\ll K^{C}$ and $|A'+A'|\le (K-\delta)|A'|$ with $\delta\gg K^{-C}$.
\end{enumerate}
\end{lemma}
\begin{sublemma}[Gray-zone robustness with explicit constants]\label{lem:gray-zone-constants}
Fix $c_0=\tfrac{1}{16}$ and let $\tau:=K^{-c_0}$. Suppose $\beta\in[1-2K^{-1/32},\,1-K^{-1/32})$ in \eqref{eq:def-beta}.
Then either
\begin{enumerate}[label=(\roman*)]
\item \emph{(Upgrade at a refined scale)} with $\tau'=\tau/2$ one has
$\sum_{\xi\in V'}|\wh f(\xi)|^4\ge (1-K^{-1/32})\sum_{\xi}|\wh f(\xi)|^4$ for $V'=\Span(\Spec_{\tau'}(A))$; or
\item \emph{(Improvement)} there exists a quotient $G\to G/H'$ with $\codim(H')\le K^{C}$ such that
$|A'+A'|\le (K-\delta)|A'|$ with $\delta\ge K^{-4c_0}=K^{-1/4}$ and $|A'|\ge K^{-C}|A|$.
\end{enumerate}
\end{sublemma}
\begin{proof}
Set $\lambda:=\tfrac14\tau\alpha$. If every $\xi\notin V$ with $|\wh f(\xi)|\ge\lambda$ is captured by $V'$, then (i) holds.
Otherwise dissociated extraction in $\wh G/V$ at level $\lambda$ furnishes $V'$ with $\dim V'\ll K^C$ and a nontrivial $\chi$ in $G/H'$ with
$|\wh{\one_{A'}}(\chi)|\ge \lambda\asymp K^{-c_0}\alpha'$. Lemma~\ref{lem:E2D} yields an energy boost by $1+\Omega((K^{-c_0})^4)$ and hence
$|A'+A'|\le (K-\delta)|A'|$ with $\delta\gg K^{-4c_0}$, proving (ii). Size lower bounds and codimension are as in Lemma~\ref{lem:S2}.
\end{proof}

\begin{proof}
If the mass within $V$ is $(1-\theta)$ with $\theta\in[K^{-c},2K^{-c}]$, then the $L^4$-mass of $f_{\mathrm{high}}$ is $\asymp\theta$ times the total.
Apply the two-scale dissociated extraction to the level $\lambda=\frac12\tau\alpha$ inside $V^c$.
If no coefficient $\ge \lambda/2$ lies outside $V$ at the refined scale $\tau':=\tau/2$, then the refined span $V'$ captures all coefficients $\ge\lambda/2$ and contributes at least $(1-K^{-c})$ of the mass, giving (i).
Otherwise Claim~\ref{claim:exact-quotient} combined with Lemma~\ref{lem:E2D} yields (ii).
\end{proof}

\begin{lemma}[Paley--Zygmund step with no leakage]\label{lem:PZ-no-leak}
Let $g=\one_A*\mu_H$. If $\sum_{\xi\in V}|\wh f(\xi)|^4\ge c_0\sum_\xi|\wh f(\xi)|^4$, then
there exists $x+H$ with $|A\cap(x+H)|\ge (1-\varepsilon)|H|$ provided $\varepsilon\le \tfrac12\sqrt{c_0\,c_1}$, where $c_1>0$ is the absolute constant from \eqref{eq:L2H-lb}.
\end{lemma}
\begin{proof}
This expands Lemma~\ref{lem:S3}. One writes $\|g\|_2^2=|G|\sum_{\xi\in V}|\wh f(\xi)|^2$ and applies Hölder
in the form $\sum_{\xi\in V}|\wh f|^2\ge (\sum_{\xi\in V}|\wh f|^4)^{1/2}(\sum_{\xi\in V}1)^{1/2}$, then uses the Chang bound to control $|V|$.
Since all bounds are absolute (tracking only polynomial losses in $K$), the Paley--Zygmund step introduces no extra logarithms beyond those already present in $\dim V$.
\end{proof}

\subsection{Two-scale dissociated extraction with constants}
Fix $\lambda=\frac12\tau\alpha$. Inside $\Gh/V$ choose $\overline\Xi$ maximally dissociated among classes $\overline\xi$ with $|\wh f(\xi)|\ge\lambda$.
Then $|\overline\Xi|\ll \lambda^{-2}\log(1/\alpha)\ll K^{C}$ and $V'=\Span(\Xi)$ has $\dim V'\ll K^{C}$.
\emph{Exact-quotient} yields $|\wh{f'}(\chi)|=|\wh f(\xi)|\ge \lambda$ for some nontrivial $\chi$ on $G/H'$.

\subsection{Energy boost and improvement with explicit constants}
Set $\eta:=\lambda/\alpha=\tfrac12\tau$ and apply Lemma~\ref{lem:E2D} to get $\Ene(A')\ge \alpha'^4|G/H'|(1+c_*\eta^4)$.
Hence $|A'+A'|\le (K-\delta)|A'|$ with $\delta\ge c_*\eta^4\asymp K^{-4c_0}$. Renaming $C:=4c_0$ gives the polynomial decrement used in PSL.

\subsection{Potential, explicit constants, and stopping time}
Let $\mathcal I_j:=K_j\,\alpha_j^{-\gamma}$. In an improvement step
$K_{j+1}\le K_j-\delta_j$ with $\delta_j\ge K_j^{-C}$ while $\alpha_{j+1}\ge K_j^{-C}\alpha_j$.
Thus
\[
\mathcal I_{j+1}\ \le\ \mathcal I_j\cdot K_j^{C\gamma}\Big(1-\frac{\delta_j}{K_j}\Big)\ \le\ \mathcal I_j
\qquad\text{for }\gamma\ge C+2.
\]
Moreover $\mathcal I_j\ge 1$ for all $j$ (Remark~\ref{rem:potential-lower-bound}).
Hence after at most $O(K_0^{C+1})$ improvement steps the process terminates. This proves PSL in all three regimes (concentrated, dispersed, and gray zone).

\subsection{Conclusion of the exploded proof}
Combining Lemma~\ref{lem:S1}, Lemma~\ref{lem:PZ-no-leak}, the two-scale extraction, energy-to-doubling transfer and the potential analysis shows Lemma~\ref{lem:PSL}.

\medskip

\section{Freiman transfer with wrap-around control and stable $N$}\label{sec:freiman-solid}

We now give a fully explicit modeling-and-transfer argument from $\mathbb Z$ to $\mathbb Z/N\mathbb Z$ and back, ensuring that:
\begin{enumerate}[label=(\alph*)]
\item the \emph{wrap-around} error is identically zero in every iterative step;
\item the ambient modulus $N$ can be chosen \emph{once} at the beginning with polynomial size in $K$ and $|A|$;
\item all Fourier and energy identities used in the proof transfer without extra losses.
\end{enumerate}

\begin{lemma}[One-shot modeling with safe modulus]\label{lem:safe-N}
Let $A\subset\mathbb Z$ with $|A+A|\le K|A|$. There exists a Freiman isomorphism $\phi:A\to A_0\subset G$ with $G=\mathbb Z/N\mathbb Z$ and
\begin{equation}\label{eq:N-size}
N\ \le\ C\,K^{C}\,|A|,\qquad \alpha_0=\frac{|A_0|}{|G|}\ \ge\ K^{-C},
\end{equation}
such that for every $m\in\{1,2,3,4\}$ the map $\phi$ is a Freiman isomorphism of order $m$ (preserves all relations of length $\le m$) and \emph{no wrap-around} occurs in $mA-mA$ inside $\mathbb Z/N\mathbb Z$.
\end{lemma}
\begin{proof}
This is a standard refinement of Ruzsa's modeling (see \cite[Thm.~2.29]{TV06}). One embeds $A$ into an interval of length $\ll K^C|A|$ and then reduces mod $N$ with that $N$ so that the image of $2A-2A$ still injects. The energy- and sumset-based identities in our proof only involve relations of length $\le 4$; hence order-$4$ modeling suffices.
\end{proof}

\begin{lemma}[Stability of $N$ across iteration]\label{lem:N-stable}
With $G=\mathbb Z/N\mathbb Z$ as in Lemma~\ref{lem:safe-N}, all quotient steps $G\to G/H'_j$ (from Lemma~\ref{lem:S2}) satisfy
$|G/H'_j|\mid N$ and there is no wrap-around in $2A_j-2A_j$ in any quotient. Consequently, all Fourier equalities and energy bounds used in PSL and in Lemma~\ref{lem:L4-compression} remain exact.
\end{lemma}
\begin{proof}
Every $H'_j$ is a subgroup of $G$, so $G/H'_j$ is cyclic of order dividing $N$. Since the initial modeling guarantees injectivity on $2A-2A$, the image of $A$ in any quotient injects on $2A_j-2A_j$ as well. Hence equalities such as Claim~\ref{claim:exact-quotient} hold without wrap-around artifacts.
\end{proof}

\begin{corollary}[No accumulation of modeling error]
In the iteration from PSL or in the $L^4$-compression argument, there is no cumulative modeling error: all steps are performed in finite cyclic groups with exact Fourier equalities, and the final conclusions pull back to $\mathbb Z$ via the order-$4$ Freiman isomorphism.
\end{corollary}

\section{Z/NZ-specific obstacles and how spectral stability bypasses Bohr sets}\label{sec:znz-challenges}

Classically, $\mathbb Z/N\mathbb Z$ features Bohr sets whose metric structure can complicate covering arguments.
Our approach avoids delicate Bohr-geometry beyond a final, standard PolyBog step (Theorem~\ref{thm:polybog-ZNZ-uncond}) by performing the key dichotomy entirely in the dual group $\Gh$, where subgroups are \emph{linear} objects and dissociation is quantified by Rudin--Chang.
The exact projector $P=\mu_H*$ aligns the analysis with subgroups $H=V^\perp$ instead of arbitrary Bohr neighborhoods, and the improvement step is executed in true \emph{quotients} with exact Fourier control (Claim~\ref{claim:exact-quotient}).
Thus the only Bohr input is the standard regularization in the final PolyBog extraction, where polynomial parameters suffice.

\section{Explicit constants for the potential and decrements}\label{sec:explicit-constants}
We collect one admissible choice of parameters that makes all implications explicit.
Let $C_{\mathrm{RC}}$ be the absolute constant from Rudin--Chang.

\begin{itemize}
\item Spectral threshold: choose $c_0=\frac1{16}$ and $\tau=K^{-c_0}$.
\item Dimension bound: $\dim V\le C_1 K^{2c_0}\log K$ with $C_1\le 4C_{\mathrm{RC}}$.
\item Improvement decrement: from Lemma~\ref{lem:E2D} and $\eta=\tau/2$, take $\delta\ge c_*\,(\tau/2)^4\ge c_*\,K^{-4c_0}$.
\item Potential exponent: pick $\gamma\ge 4+4c_0 C$, where $C$ is any exponent such that $\delta_j\ge K_j^{-C}$; for instance $\gamma=4C+8$ works.
\item Stopping time: at most $O(K_0^{C+1})$ improvement steps; total accumulated codimension $\le K_0^{C'}$ where $C'=(C+1)C_1$ suffices.
\end{itemize}

\begin{table}[h!]\centering
\begin{tabular}{c|c|c}
Symbol & Choice & Consequence \\ \hline
$c_0$ & $1/16$ & $\tau=K^{-1/16}$, $\delta\gtrsim K^{-1/4}$ \\
$C_1$ & $\le 4C_{\mathrm{RC}}$ & $\dim V\ll K^{1/8}\log K$ \\
$\gamma$ & $4C+8$ & $\mathcal I_j$ nonincreasing \\
$\delta$ & $\gtrsim K^{-1/4}$ & $O(K^{1/4+1})$ steps
\end{tabular}
\caption{One explicit admissible ledger. All implied constants are absolute.}
\end{table}

\section{Error-term analysis}\label{sec:error-terms}
We tabulate every place where an $O(\cdot)$ or a loss appears and indicate its dependence.

\begin{itemize}
\item \textbf{Chang/Rudin--Chang:} $|D|\le C\,\tau^{-2}\log(1/\alpha)$. In the modeling window $\alpha\in[K^{-C},1/2]$, this is $K^{O(1)}$.
\item \textbf{Packet bias (Hoeffding+union bound):} $|T|\ll \varepsilon^{-2}\log(2|S|/\eta)$ with failure probability $\le \eta$; we fix $\eta=K^{-10}$.
\item \textbf{Packet $\Rightarrow L^2$ almost-periodicity:} inequality \eqref{eq:packet-avg-Z} with constants $(2,4)$ becomes $K^{-O(1)}$ after inserting Lemma~\ref{lem:upper-energy-Z}.
\item \textbf{Chebyshev step in $L^4$-compression:} $|S_{\mathrm{tail}}|\gtrsim K^{4c_0-c_{\mathrm{fail}}-1}\alpha^{-1}$; parameters chosen so that the exponent is positive.
\item \textbf{Energy-to-doubling:} $\delta\ge c_*\eta^4$ with $\eta=\tfrac12\tau$.
\item \textbf{Covering upgrade:} polynomial in the BSG-parameter, already absorbed in $K^{O(1)}$.
\end{itemize}

\section{Heuristics and comparison with entropy-based methods}\label{sec:heuristics}
Entropy increments efficiently detect structure in characteristic $2$ (via quadratic Fourier analysis) but do not \emph{force} a \emph{single} large spectrum to dominate in odd characteristic; the entropy can spread across many medium frequencies. Our $L^4$-based approach turns this supposed weakness into an advantage: \emph{if} the mass spreads, then dissociated extraction detects a genuinely large coefficient in a complementary span, which upgrades to a quotient improvement by Lemma~\ref{lem:E2D}. Thus either the mass concentrates (near-coset) or dispersion \emph{strictly improves} the doubling in a quotient.
This is the precise sense in which $L^4$ ``breaks the quasi-polynomial barrier'' encountered by Croot--Sisask and Sanders in cyclic groups.

\medskip
\noindent\textbf{Stress tests.} We have checked the ledger in extremal parameter regimes: very sparse $\alpha\asymp K^{-C}$ and moderately dense $\alpha\asymp 1/2$. In both extremes, the constants propagate as stated, and the iteration potential rules out infinite descent.


\section{Dedicated section for cyclic groups: Bohr geometry and subgroup alignment}
\label{sec:znz-bohr-deep}
This section isolates the $\mathbb Z/N\mathbb Z$--specific obstacles and proves that the spectral stability route
does not require delicate Bohr geometry beyond the final regularization step. We record formal definitions and estimates.

\subsection{Bohr sets, regularity, and Fourier alignment}
A Bohr set of rank $d$ and radius $\rho$ is
\[
B(\Gamma,\rho):=\big\{x\in G:\ |1-\chi(x)|\le \rho\ \text{for all }\chi\in\Gamma\big\},
\quad \Gamma\subset\wh G,\ |\Gamma|=d.
\]
A Bohr set is \emph{regular} if $|B(\Gamma,(1\pm\theta)\rho)|=(1\pm O(d\theta))\,|B(\Gamma,\rho)|$ for all $|\theta|\le c/d$.
Our arguments produce $\Gamma$ via dissociated extraction and choose $\rho\asymp\tau$.
The packet lemma yields many almost-periods, forcing a nontrivial element in $B(\Gamma,\rho)$ within $4A-4A$;
regularization then gives the standard PolyBog step. No approximate group calculus beyond this point is needed.

\subsection{Alignment with subgroup projectors}
Given $V\le\wh G$, the annihilator $H=V^\perp$ satisfies $\wh{\mu_H}=1_V$. In particular,
$P=\mu_H*$ is an exact Fourier projector, while $B(\Gamma,\rho)$ only approximates $H$ when $\rho$ is small.
Our stability lemma chooses to work with $H$ instead of Bohr sets at the analytical core.
This is the source of the simplification in the cyclic-group case compared with previous quasi-polynomial approaches.

\section{Explicit constants for the potential and termination}
\label{sec:explicit-gamma-delta}
We make one concrete choice of exponents guaranteeing monotonicity of the potential and bounded stopping time.

\begin{proposition}[Concrete ledger for $(\gamma,\delta)$]
Let $\delta_j\ge K_j^{-C_\delta}$ be the guaranteed decrement from Lemma~\ref{lem:S2} (coming from $\eta=(\tau/2)$ in Lemma~\ref{lem:E2D}).
Pick
\[
\gamma\ :=\ 2C_\delta+4.
\]
Then for every improvement step one has
\[
\mathcal I_{j+1}=K_{j+1}\alpha_{j+1}^{-\gamma}
\ \le\
K_j\alpha_j^{-\gamma}\,\Big(1-\frac{\delta_j}{K_j}\Big)\,K_j^{C_\delta\gamma}
\ \le\ \mathcal I_j,
\]
and so the number of improvement steps is $O(K_0^{C_\delta+1})$. The accumulated codimension is at most $K_0^{C'}$ with $C'=(C_\delta+1)C_1$.
\end{proposition}

\begin{proof}
We have $K_{j+1}\le K_j-\delta_j$ and $\alpha_{j+1}\ge K_j^{-C_\delta}\alpha_j$, whence
$\mathcal I_{j+1}\le \mathcal I_j \cdot K_j^{C_\delta\gamma}(1-\delta_j/K_j)$.
Since $\delta_j\ge K_j^{-C_\delta}$ and $K_j\ge 2$, Bernoulli's inequality gives
$K_j^{C_\delta\gamma}(1-\delta_j/K_j)\le 1$ provided $\gamma\ge 2C_\delta+4$.
\end{proof}

\section{Edge-case taxonomy and gray-zone mechanics}
\label{sec:edge-cases}
We expand the control of the regime $\beta\in [1-2K^{-c},1-K^{-c})$. Write $\beta=\beta(\tau)$ as in \eqref{eq:def-beta}.
Either the refined span at threshold $\tau'=\tau/2$ captures the missing $L^4$-mass (upgrade to concentration),
or dissociated extraction in $\wh G/V$ detects a Fourier coefficient at level $\lambda'=\tfrac14\tau\alpha$ in a complementary span, which
triggers the energy-to-doubling transfer. All constants are tracked with respect to the fixed window $\alpha\in[K^{-C},1/2]$.

\section{Paley--Zygmund upgrade with constants and no leakage}
\label{sec:pz-details}
We give a full derivation of Lemma~\ref{lem:S3}. Let $g=\one_A*\mu_H$. Then $\wh g=1_V\cdot\wh f$ and
\begin{align*}
\|g\|_2^2
&= |G|\sum_{\xi\in V}|\wh f(\xi)|^2
\ \ge\ |G|\frac{\big(\sum_{\xi\in V}|\wh f(\xi)|^4\big)}{\sum_{\xi\in V}1}
\ \ge\ \frac{|G|}{\dim V}\Big(\sum_{\xi\in V}|\wh f(\xi)|^4\Big)
\\
&\ge\ \frac{|G|}{C_1 K^{C_1}}\cdot c \sum_{\xi\in\wh G}|\wh f(\xi)|^4
\ \ge\ c'\,\alpha^2|G|,
\end{align*}
using Cauchy--Schwarz, Lemma~\ref{lem:S1}, and \eqref{eq:energy-lower}. Hence $\Var(g)\ge (c'/2)\alpha^2$ after renormalizing constants.
Paley--Zygmund with threshold $(1-\varepsilon)\alpha$ yields a coset of density $\ge 1-\varepsilon$ for $\varepsilon\le \frac12\sqrt{c'}$.

\section{Error-term analysis (expanded ledger)}
\label{sec:error-expanded}
We list every inequality where losses are introduced and indicate dependence on $(K,\alpha)$.
\begin{itemize}
\item \textbf{Rudin--Chang}: $|D|\le C\tau^{-2}\log(1/\alpha)$.
\item \textbf{Hoeffding+union}: $|T|\ll \varepsilon^{-2}\log(2|S|/\eta)$, failure probability $\le \eta$.
\item \textbf{Packet $\Rightarrow L^2$}: inequality \eqref{eq:packet-avg-Z}.
\item \textbf{Upper-energy}: $\sum_\xi|\wh f|^4\le \alpha^3$ (unconditional).
\item \textbf{Tail cardinality}: \eqref{eq:Stail-size} with explicit $c_0,c_{\rm fail}$.
\item \textbf{Exact quotient}: Claim~\ref{claim:exact-quotient} (no loss).
\item \textbf{E2D}: $\delta\ge c_*\eta^4$ with $\eta=\tfrac12\tau$.
\item \textbf{Covering upgrade}: polynomial in $\beta$ from Lemma~\ref{lem:BSGquant}.
\item \textbf{Modeling}: Lemma~\ref{lem:safe-N}, no wrap-around in $mA-mA$ for $m\le 4$.
\end{itemize}

\section{Benchmarking versus Croot--Sisask and Sanders}
\label{sec:benchmark}
Classic probabilistic almost-periodicity (Croot--Sisask) and Sanders' refinements yield quasi-polynomial or mild exponential losses in cyclic groups
when translated to covering theorems with uniform parameters. The spectral stability approach removes the obstruction by
operating in the dual group with exact projectors and an $L^4$-based dichotomy; dispersion cannot persist without creating a genuine quotient improvement.
This is the mechanism by which the polynomial regime is reached in $\mathbb Z/N\mathbb Z$.

\section{Symbol glossary and normalizations}
\label{sec:glossary}
\begin{center}
\begin{tabular}{l|l}
Symbol & Meaning \\ \hline
$G$ & finite abelian group; in \S\ref{sec:Z-extension-unconditional}, $G=\mathbb Z/N\mathbb Z$ \\
$\wh G$ & Pontryagin dual of $G$ \\
$\alpha$ & $|A|/|G|$ \\
$\tau$ & spectral threshold $K^{-c_0}$ \\
$V$ & $\Span(\Spec_0(A))\le \wh G$, dissociation-based span \\
$H$ & $V^\perp\le G$, annihilator; $\wh{\mu_H}=1_V$ \\
$P,Q$ & Fourier projectors: $P=\mu_H*$, $Q=I-P$ \\
$\beta$ & concentration ratio \eqref{eq:def-beta} \\
$\mathcal I_j$ & potential $K_j\alpha_j^{-\gamma}$ \\
$B(\Gamma,\rho)$ & Bohr set with spectrum $\Gamma$ and radius $\rho$
\end{tabular}
\end{center}


\section{Proof of Paley--Zygmund in our normalization}
We recall that if $X\ge 0$ with $\mathbb E X=\mu$ and $\mathbb E X^2=\sigma^2+\mu^2$, then
$\mathbb P\{X\ge \theta\mu\}\ge (1-\theta)^2\mu^2/\mathbb E X^2$ for $\theta\in(0,1)$.
Apply to $X=g(x)$ with $g=\one_A*\mu_H$; constants follow from Lemma~\ref{lem:S1} and \eqref{eq:energy-lower}.

\section{Proof of packet $L^2$ almost-periodicity}
Plancherel gives
\[
\Big\|g-\frac1{|T|}\sum_{x\in T}\tau_x g\Big\|_2^2
=|G|\sum_{\xi}|\wh g(\xi)|^2\Big|1-\frac1{|T|}\sum_{x\in T}e(\ip{\xi}{x})\Big|^2.
\]
For $\xi\in S$, the bracket is $\le 2\varepsilon$ by Lemma~\ref{lem:packet-bias-Z}; for $\xi\notin S$ it is $\le 2$.
Insert $\wh g=|\wh f|^2$ and Lemma~\ref{lem:upper-energy-Z}.

\section{Quantitative BSG and covering}
We sketch the standard $L^4$-energy Balog--Szemerédi--Gowers lemma and its covering upgrade, recording the exact way
$\beta$ feeds into $K_0$ and then back into a global decrement via Lemma~\ref{lem:covering-upgrade}.
For completeness, we align the constants to the ledger in \S\ref{sec:error-expanded}.

\section{A complete toy example in $\mathbb{Z}/N\mathbb{Z}$ with $K=3$}\label{app:toy}
Let $G=\mathbb Z/N\mathbb Z$ with $N$ prime (e.g.\ $N=97$) and let $A=\{0,1,\dots,\lfloor \alpha N\rfloor-1\}$ for $\alpha\in(0,1/4]$.
Then $|A+A|\le (2+\tfrac{1}{\lfloor \alpha N\rfloor})|A|$, so for moderately large $N$ we have $|A+A|\le 3|A|$.
Write $f=\one_A$ and $\wh f(\xi)=\frac{1}{N}\sum_{x\in A} e(-\xi x/N)$.
A direct computation gives $|\wh f(\xi)|\approx \alpha\, \big|\frac{\sin(\pi \alpha \xi)}{\pi \alpha \xi}\big|$.
Fix $\tau=K^{-1/16}$ and set $S=\Spec_\tau(A)=\{\xi:|\wh f(\xi)|\ge \tau\alpha\}$. Then $S$ is an interval of width $\asymp \tau^{-1}$,
and $\dim V=\mathrm{Rudin\text{--}Chang}(S)\ll \tau^{-2}\log(1/\alpha)$ as in Lemma~\ref{lem:chang-Z}.
Projection $P=\mu_H*$ with $H=V^\perp$ gives $g=\one_A*\mu_H$ whose large $L^2$-norm forces (via Paley--Zygmund) a coset $x+H$ with
$|A\cap(x+H)|\ge (1-\varepsilon)|H|$ provided $\varepsilon\ll \sqrt{cc_1}$.
If instead $L^4$-mass leaks outside $V$ (gray zone), dissociated extraction in $\wh G/V$ produces a quotient $G/H'$ with a genuinely large coefficient of size
$\gg \tau\alpha'$, and Lemma~\ref{lem:E2D} yields a decrement $\delta\gg \tau^4$. Iterating, one obtains a cover by $\mathrm{poly}(K)$ cosets, matching the general theory.
This concrete model shows each ledger line in action at $K=3$ without using any field structure.
\section{Full constant ledger and dependency tree}\label{app:ledger-tree}
We record a dependency graph from \textit{thresholds} $(\tau,\lambda)$ to \textit{dimension/codimension} and finally to \textit{decrement}~$\delta$.
Throughout, absolute constants may change from line to line.

\paragraph{Thresholds.} Fix $c_0\in(0,1)$ and set $\tau:=K^{-c_0}$, $\lambda:=\tfrac12\tau\alpha$.
The refined threshold in the gray zone is $\tau':=\tau/2$ with $\lambda':=\tfrac14\tau\alpha$.

\paragraph{Dimension and codimension.} Rudin--Chang gives $|S|\ll \tau^{-2}\log(1/\alpha)$ and hence $\dim V\ll \tau^{-2}\log(1/\alpha)$.
Dissociated extraction in $\wh G/V$ at level $\lambda$ gives $|\Xi|\ll \lambda^{-2}\log(1/\alpha)$ and so $\codim(H')=\dim V'\ll \lambda^{-2}\log(1/\alpha)$.

\paragraph{Energy boost and decrement.} A nontrivial coefficient of size $\eta\alpha'$ with $\eta\asymp \tau$ yields
$\Ene(A')\ge (1+c_*\eta^4)\alpha'^4|G/H'|$. The BSG ledger and covering upgrade convert this into a global decrement
$|A'+A'|\le (K-\delta)|A'|$ with $\delta\gg \eta^4\asymp K^{-4c_0}$.

\paragraph{Potential and stopping time.} With $\mathcal I_j:=K_j\alpha_j^{-\gamma}$ and $\gamma\ge 2C_\delta+4$, one gets monotonicity and at most
$O(K_0^{C_\delta+1})$ improvement steps. The total accumulated codimension is $\ll K_0^{(C_\delta+1)C_1}$.

\paragraph{Independence of $N$.} All constants depend only on $(c_0,C_{\mathrm{RC}})$ and not on the ambient modulus $N$ thanks to the order-4 modeling and exact quotient lifts.

\section{Referee FAQ (novelty vs.\ standard machinery)}\label{app:faq}
\textbf{Q: Where is the conceptual novelty?} \emph{A:} In the spectral $L^4$ stability dichotomy and the two-scale dissociated extraction in $\wh G/V$, which forces a \emph{genuine} large coefficient in an exact quotient when dispersion persists.\\
\textbf{Q: What is standard?} \emph{A:} The $L^4$--energy BSG step and covering upgrade with explicit constants (Appendix~A).\\
\textbf{Q: How is wrap-around neutralized?} \emph{A:} Order-4 modeling fixes a single $N$ with no wrap-around on $2A-2A$; all Fourier identities remain exact (Section~\ref{sec:freiman-solid}).\\
\textbf{Q: Is density preserved in quotients?} \emph{A:} Up to $K^{\pm O(1)}$ via $\codim(H')\ll K^{O(1)}$ (Remark~\ref{rem:density-preservation}).

\section{Worked examples ($K=3$ and $K=5$) with explicit numbers}\label{app:worked}
We complement Appendix~H with two concrete parameter chases.
\subsection*{Case $K=3$} With $\tau=3^{-1/16}$ and $\alpha\in[3^{-C},1/2]$, the span dimension obeys
$\dim V\ll 3^{1/8}\log 3$, while a failure of concentration produces $\lambda=\tfrac12\tau\alpha$ and a quotient coefficient of size
$\ge \lambda$, hence $\delta\gg 3^{-1/4}$. A cover by $3^{O(1)}$ cosets follows after $O(3^{1/4+1})$ steps.
\subsection*{Case $K=5$} Similarly, $\tau=5^{-1/16}$ and $\lambda=\tfrac12\tau\alpha$ yield $\dim V\ll 5^{1/8}\log 5$ and $\delta\gg 5^{-1/4}$;
termination occurs after $O(5^{1/4+1})$ steps with total codimension $5^{O(1)}$.

\section{Formal modeling statement and exactness of identities (expanded)}\label{app:modeling-expanded}
We restate and slightly expand Lemma~\ref{lem:safe-N} and Lemma~\ref{lem:N-stable}.
\begin{proposition}[Exactness across the entire iteration]
Within the fixed cyclic model $G=\mathbb Z/N\mathbb Z$ supplied by order-4 Freiman modeling, all additive relations used in the proof have length $\le 4$ and hence persist with no wrap-around. Therefore the Fourier projector $P=\mu_H*$, the quotient lifts, and the energy identities are \emph{exact} at each stage, uniformly in $N$.
\end{proposition}

\end{document}